\documentclass[11pt,authoryear]{elsarticle}
\usepackage{fullpage}

\usepackage{amssymb}
\usepackage{lscape}
\usepackage{amsthm}
\usepackage{lineno}

\usepackage{placeins, microtype} 
\usepackage{amsmath}
\usepackage{cases}
\usepackage[hidelinks]{hyperref}
\usepackage{mathtools}
\usepackage{color}
\usepackage{calligra}
\usepackage{xspace}
\usepackage{amssymb}
\usepackage[rgb,dvipsnames]{xcolor}
\usepackage{caption}
\usepackage{subcaption}
\usepackage{comment}
\usepackage{hyperref}
\usepackage{algpseudocode}
\usepackage{algorithm}
\usepackage{wasysym}

\newtheorem{thm}{Theorem}
\newtheorem{lemma}{Lemma}

\newtheorem{corollary}{Corollary}
\newtheorem{example}{Example}

\def\Z{\mathbb{Z}}

\def\P{{\mathbb P}}     
\def\E{{\mathbb E}}

\definecolor{gr}{rgb}{0.03, 0.7, 0.29}

\journal{arXiv}

\begin{document}

\begin{frontmatter}



\title{Correlation of coalescence times in a diploid Wright-Fisher model with recombination and selfing}


\author[1]{David Kogan\texorpdfstring{\corref{equal}}}
\ead{dkogan1@sas.upenn.edu}
\author[2]{Dimitrios Diamantidis\texorpdfstring{\corref{equal}}}
\ead{didiaman@iu.edu}
\author[3]{John Wakeley}
\ead{wakeley@fas.harvard.edu}
\author[2,3]{Wai-Tong (Louis) Fan\texorpdfstring{\corref{cor}}}
\ead{waifan@iu.edu}
\cortext[equal]{These authors contributed equally to this work}
\cortext[cor]{Corresponding author}

\affiliation[1]{organization={Department of Mathematics, University of Pennsylvania},
            addressline={209 South 33rd Street}, 
            city={Philadelphia},
            postcode={19104}, 
            state={PA},
            country={USA}}

\affiliation[2]{organization={Department of Mathematics, Indiana University},
            addressline={831 East 3rd St}, 
            city={Bloomington},
            postcode={47405}, 
            state={IN},
            country={USA}}

\affiliation[3]{organization={Department of Organismic and Evolutionary Biology, Harvard University},
            addressline={16 Divinity Ave}, 
            city={Cambridge},
            postcode={02138}, 
            state={MA},
            country={USA}}

\begin{abstract}
The correlation among the gene genealogies at different loci is crucial in  biology, yet challenging to understand because
such correlation depends on many factors including genetic linkage, recombination, natural selection and population structure.
Based on a diploid Wright-Fisher model with a single mating type and partial selfing for a constant large population with size $N$, we quantify the combined effect of genetic drift and two competing factors, recombination and selfing, on the correlation of coalescence times at two linked loci for samples of size two.
Recombination decouples the genealogies at different loci and decreases the correlation while selfing increases the correlation. We obtain explicit asymptotic formulas for the correlation for four scaling scenarios that depend on whether the selfing probability and the recombination probability are of order $O(1/N)$ or $O(1)$ as $N$ tends to infinity.
Our analytical results 
confirm that the asymptotic lower bound in [King, Wakeley, Carmi (TPB 2018)] 
is sharp when the loci are unlinked and when there is no selfing,  
and provide a number of new formulas for other scaling scenarios that have not been considered before. 
We present asymptotic results for the variance of Tajima's estimator
of the population mutation rate for infinitely many loci
as $N$ tends to infinity.  When the selfing probability is of order 
$O(1)$ and is
equal to a positive constant $s$ for all $N$ and if the samples at both loci are in the same individual, then the variance of the Tajima's estimator tends to $s/2$ (hence remains positive) even when  the recombination rate, the number of loci and the population size all tend to infinity. 
\end{abstract}




\begin{keyword}

Recombination \sep selfing  \sep Wright-Fisher model \sep diploid population \sep coalescent \sep  asymptotic analysis



\end{keyword}

\end{frontmatter}


\section{Introduction} \label{sec:intro}

Population geneticists make frequent use of stochastic models to represent the dynamics and evolution of gene frequencies in finite populations. The most classical models are the Wright-Fisher model \citep{Fisher1922,Fisher1930b,Wright1931} and the Moran model 
\citep{moran_1958,Moran1962}, 
which differ in their treatment of generations —discrete and non-overlapping in the former, and overlapping in the latter. 
Many extensions and  variations of such models were developed over the past century to incorporate selection, mutation, general offspring distributions, varying population size, spatial movement and other details. Such models include the Cannings exchangeable models \citep{cannings1,cannings2} for haploid populations, measure-valued processes \citep{ethier1993fleming}, exchangeable diploid models \cite{BirknerEtAl2018}.  
For detailed overviews, see \citet{Ewens2004} and a  recent paper by  \citet{etheridge2019genealogical}.

While classical population genetics theory focuses on forward-in-time models, 
current theory emphasizes a retrospective approach  which focuses on the backward-in-time ancestry of a sample of homologous gene copies, alleles or halpotypes from a population. This shift of focus to a retrospective approach is motivated by the relevance of coalescent theory to genomics data and statistical inference. 
Since the seminal work of \cite{Kingman1982235,Kingman1982b},
the Kingman coalescent is widely used in the description of gene genealogies 
and in the derivation of sampling probabilities \citep{Hudson1983a,Hudson1983b,Tajima1983, Ewens1990}.
Mathematically, it was shown that the Kingman coalescent arise as a robust scaling limit  under the Cannings model for general offspring distribution that is is not too skewed  \citep{Mohle1998a,Mohle1998c}. In the past decades, other scaling limits such as the multiple merger coalescent and simultaneous multiple merger coalescent were deduced from exchangeable population models. See  \cite{wakeley2009coalescent} for an introduction  and \cite{berestycki2009recent} for a probabilistic treatment on coalescent theory.

The genealogies of a population 
trace the evolutionary relationships among haplotypes (alleles) in the population. They are influenced  by many  factors including selection, fluctuation in population size,  population substructure and various reproduction mechanisms. 
These factors can interact and influence each other, making the correlation among genealogies complex and variable across different loci and populations.
Here, we investigate the {\it combined effects} of recombination and partial selfing on the correlation of pairwise coalescence times. 
While the separate effect of recombination and of selfing on  gene genealogies 
have been extensively studied, their joint effects are much less explored \citep{nordborg2000linkage}. 
Which factor has a stronger effect on determining the correlation? How do these two competing factorss act together to shape the gene genealogies across loci? 
We examine  these questions by focusing on a sample of size two at two loci under four different scenarios, or scaling regimes, for the relative strengths of recombination and selfing.  Namely, we consider Scenario $(i)$: $(s_N,r_N)=(\tilde{\sigma}/N, \tilde{\rho}/N)$,  Scenario $(ii)$:  $(s_N,r_N)=(s, \tilde{\rho}/N)$, Scenario $(iii)$: $(s_N,r_N)=(\tilde{\sigma}/N, r)$, and Scenario $(iv)$: $(s_N,r_N)=(s, r)$, where $s,r,\tilde{\sigma},\tilde{\rho}$ are constants that do not depend on $N$.

The gene genealogies, even at unlinked loci, are correlated because all loci share  the same underlying pedigree \citep{king2018non, Bursts_DFBW2023}.  Note, with respect to the approach taken in \citet{Bursts_DFBW2023}, here we obtain results by averaging over outcomes of reproduction rather than by conditioning on the pedigree of the population, so our results may be compared with previous findings.  We compare our results with those of \cite{king2018non}, who derived an analytical formula for the correlation coefficient for two unlinked loci under a 2-sex diploid discrete-time Wright-Fisher model (DDTWF), in terms of the  sampling configuration and the finite population size $N$. They obtained an asymptotic lower bound of $1/(12 N)$ for the correlation coefficient; see
equations (14) and (5) of \cite{king2018non}. This quantitative bound for the correlation implies an approximate  
positive lower bound for the variance of Tajima's estimator for infinitely many loci in a large population of \textit{finite} size $N$, even when  the number of independently segregating loci approaches infinity.


We build on the work of \cite{king2018non} by including  partial selfing in addition to recombination, and we significantly broaden the asymptotic regimes  in the parameter space by 
establishing asymptotics for the correlation in four scenarios depending on whether the selfing probability $s_N$ and the recombination probability $r_N$ are of order $O(1/N)$ or order $O(1)$ as $N\to\infty$. We do this by describing a Markov chain for the two-locus ancestral process, where the configuration of the sample gives the initial state; twelve are possible. We obtain the correlations for all twelve and highlight three that are most biologically relevant, including the one we denote $q_{12}$ which corresponds to sampling two pairs of gene copies on both   chromosomes of a single individual. This initial state was considered in \cite{king2018non}.
The aforementioned result in \cite{king2018non} corresponds to the case when the  recombination probability $r_N=1/2$, selfing probability $s_N=0$ and the initial condition is $q_{12}$ in this paper. 
Indeed, in equation \eqref{GeneralizedKingetal} below,  we generalize this result to incorporate partial selfing and we show that the factor $1/12$ in the above lower bound is sharp in the sense that the asymptotic is equal to $1/(12N)$ as $N\to\infty$. Our result \eqref{GeneralizedKingetal} therefore also implies a matching asymptotic upper bound for the variance of Tajima's estimator. Furthermore, 
we find that when  the per-generation selfing probability $s_N=s$ for all $N$ (in particular, of order $O(1)$),
the asymptotic correlation coefficients 
stays positive and tends to $s/2$, even when the recombination rate tends to infinity.
This implies that the variance of the Tajima's estimator for infinitely many loci stays positive, even when the recombination rate and the population size $N$ both tend to infinity; see \eqref{Positive_Var} below. This strengthens the previous conclusion that Tajima's estimator is not consistent. 








Previous work on recombination and partial selfing includes \citet{GoldingAndStrobeck1980} who obtained probabilities of identity by descent for samples of size two at two linked loci using single-generation recursive equations, assuming infinite-alleles mutation and reproduction in a finite Wright-Fisher population with partial selfing.  For our Scenario $(ii)$, they discovered a correspondence with the same probabilities of identity by descent previously known for a randomly mating population \citep{StrobeckAndMorgan1978} that is with an effective population size $N_e=N(2-s)/2$ and an effective mutation rate $r_e=r2(1-s)/(2-s)$, so $N_er_e=N(1-s)$.  \citet{vitalis_couvet_2001} studied how population subdivision and migration affect such identity probabilities.  \citet{nordborg2000linkage} reframed and extended the ideas in \citet{GoldingAndStrobeck1980} to describe a coalescent process for larger samples, with the same re-scaling of recombination and with a coalescent rate $2/(2-s)$ times faster than that for a randomly mating population of the same size.  \citet{pollak1987theory} established $N_e=N(2-s)/2$ for single-locus evolution forward in time and \citet{Mohle1998a} proved convergence to the corresponding Kingman coalescent process with time measured in units of $2N_e$ generations.  We do not consider selection here but we note that, because the effective recombination rate depends on $s$, a number of authors have studied the evolutionary interactions of recombination, selfing and selection \citep{roze2005self,Roze2015,Hartfield2016,Roze2016,RYBNIKOV2021110849,Rybnikov2021,roze2022,Sianta2023}.


The structure of this paper is as follows.  In Section \ref{sec:Models}, we describe first the discrete diploid Wright-Fisher model and then the coalesce times ($T_i$ and $T_j$) at the two loci. In Section \ref{S:Tresults},
we first recall single-locus asymptotic results and the effect of partial selfing, then state our main analytical results for the covariance of $T_i$ and $T_j$. In Section  \ref{S:interpretation} we consider the correlation coefficients in Corollary \ref{C:q12q11} and  offer some interpretation of the complicated formulas, and we mention their implication on the Tajima's estimator.
In Section \ref{S:simulations} we present our Monte Carlo simulations that confirm our results from Corollary \ref{C:q12q11}. The discussion Section \ref{S:Discussion} focuses on what may be lacking in this paper and implications for future work. Proofs of our analytical results are in the Appendix.

\section{Models and methods} \label{sec:Models}  

\subsection{The population model} \label{S:Model} 

We consider a one-sex diploid Wright-Fisher model of constant size $N$, in which the population consists of discrete non overlapping generations. 
We model the ancestry of a random sample of size two at each of two loci.
There are two parameters, $s_N\in [0,1]$ and $r_N\in[0,1]$, called  the selfing probability and the recombination probability respectively. 
The population dynamics is as follows:
\begin{enumerate}
    \item (Parents assignment) Individuals in past generation $g+1$ are parents of individuals in past generation $g$. At generation $g\in \mathbb Z_+=\{0,1,2,\cdots\}$, each child  selects either one parent with probability $s_N$ or two distinct parents uniformly from generation $g+1$ with probability $1-s_N$.
    \item (Chromosomes of a child) The chromosomes of each child are determined by Mendelian inheritance from the chromosomes of the parent(s).  Each chromosome of a child is a non-recombinant with probability $1-r_N$, and is a recombinant with probability $r_N$, independent of any other chromosomes. If it is a non-recombinant,
 it is uniformly chosen from the two parental chromosomes; if it is a recombinant, then one locus came from a uniformly chosen parental chromosome and the other locus came from the other parental chromosome.    

\end{enumerate}

For example, consider the event that the child is produced by selfing and there is no recombination for either chromosome of the child (top-left case in \textbf{Figure \ref{Fig:Chrom+loci}}). The probability of this event is
$s_N (1-r_N)^2$. Furthermore, given the occurrence of this event, each chromosome of the child picks a parental chromosome uniformly at random, independently; hence coalescent of the two chromosomes of the child occurs in one generation with conditional  probability $1/2$. 

\FloatBarrier
\begin{figure}[ht]
    \centering
    \includegraphics[scale=0.22]{ 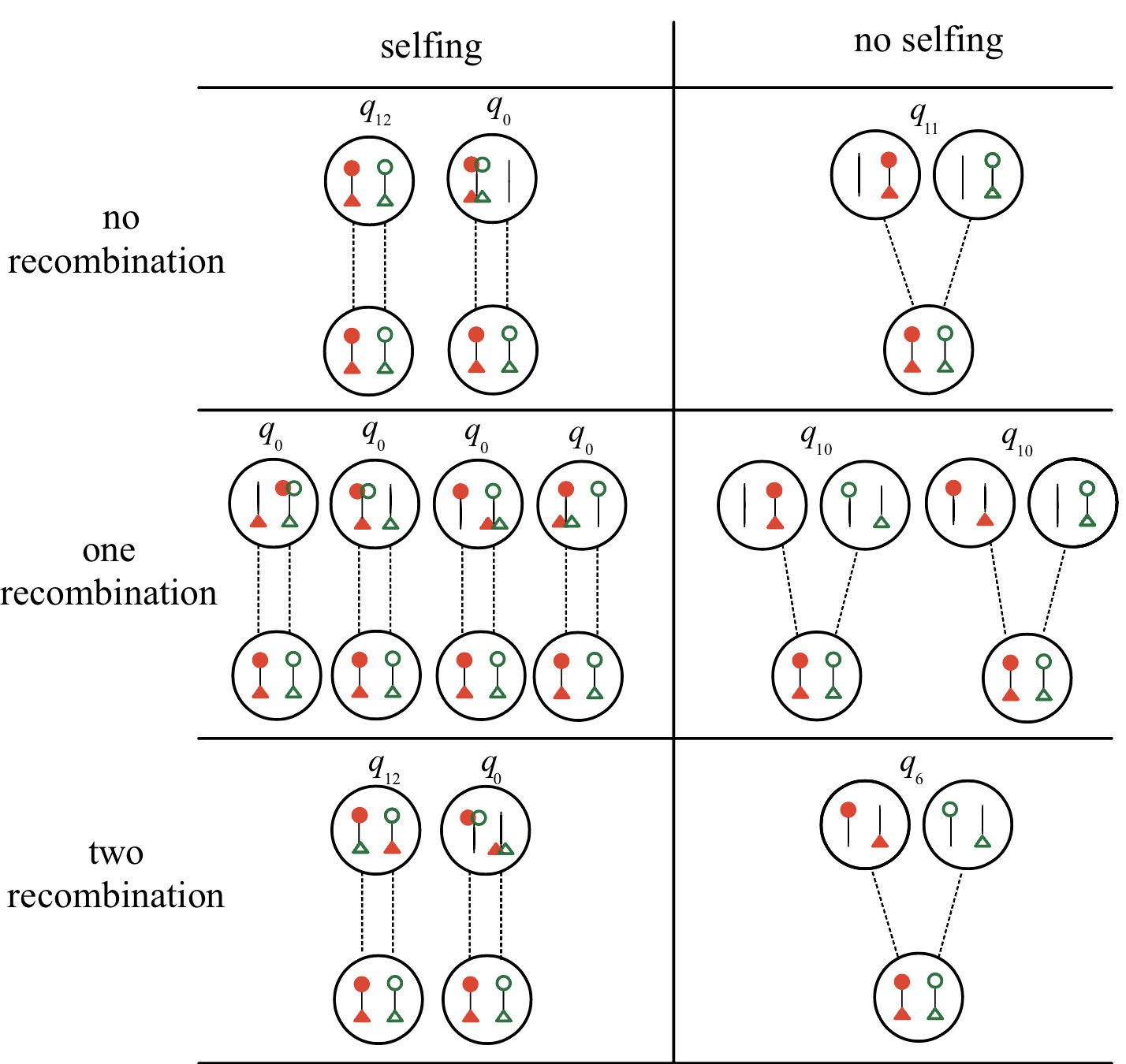}
    \caption{All possible one-step transitions, backward-in-time, for two pairs of gene copies at respectively two loci  for the child (the bottom big circle), starting at state $q_{12}$ (see \ref{A:States}).  These possibilities are organized into 6 cells with 3 cells in each of the left column and the right column. All possibilities within each cell are equality probable.
    [{\bf Left Column}] All possibilities given that selfing occured for the child. The 3 events have probabilities $s_N(1-r_N)^2$, $2s_Nr_N(1-r_N)$ and $s_Nr_N^2$ respectively from top to bottom. [{\bf Right Column}] All possibilities given that the child has two parents.  The 3 events have probabilities $(1-s_N)(1-r_N)^2$, $2(1-s_N)r_N(1-r_N)$ and $(1-s_N)r_N^2$ respectively from top to bottom.     
    }
    \label{Fig:Chrom+loci}
\end{figure}


\subsection{Ancestral process  at two loci} \label{S:ancrestralProcess} 

\FloatBarrier


We consider two loci $i$ and $j$. For each $k\in\{i,j\}$, we consider two distinct gene copies  and we let $T_k$ be the coalescence time for the gene copies at locus $k$, measured in generations. 

The ancestral process can be represented by a discrete-time Markov chain whose state at time $g\in \Z_+$ is the state of the two gene copies at generations $g$ in the past. The fact that it is Markovian follows from the argument in \cite[Appendix]{Bursts_DFBW2023}. 
Let $q_0$ be the state in which coalescence occurs at one or both of the loci. Then  there are  13 possible states $\{q_i\}_{i=0}^{12}$ and they are listed in \ref{A:States}.
The transition probabilities of this Markov chain among the states  $\{q_i\}_{i=0}^{13}$ are computed in \ref{A:Matrix}.

\textbf{Figure \ref{Fig:TiTj}} illustrates a realization of the ancestral process starting at the state $q_{12}$. In this figure, $T_i=2$ and $T_j=3$, where  $\bf{\color{red}\bullet}$ and ${\color{OliveGreen}\circ}$ represent the two gene copies at locus $i$, and ${\color{red}\blacktriangle}$ and  ${\color{OliveGreen}\triangle}$ represent the two gene copies at  locus $j$.

    \begin{figure}[htp]
    \begin{center}
        \includegraphics[scale=0.1]{ 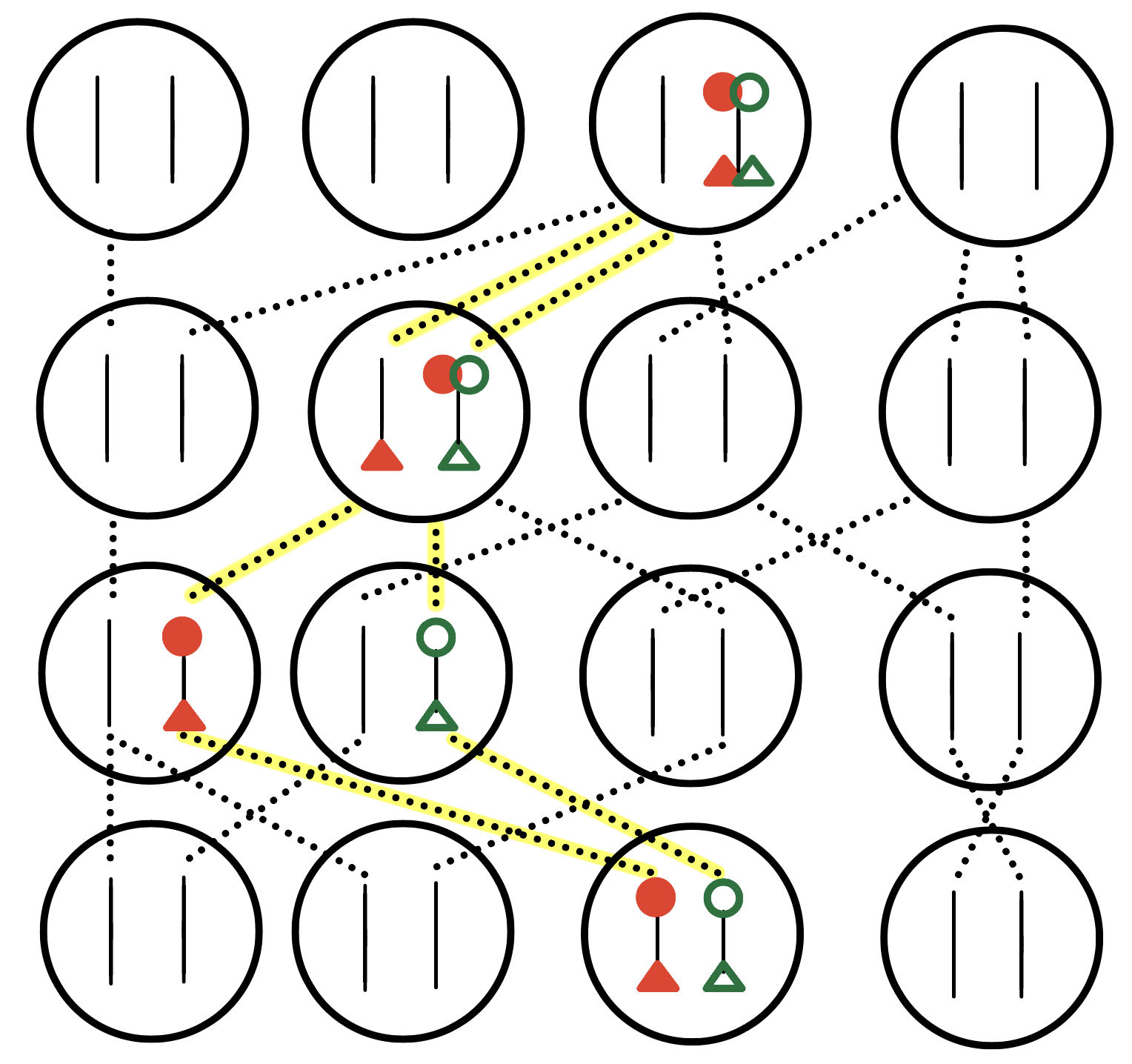}
    \end{center}
    \caption{The two gene copies at locus $i$, represented by the small circles, coalesce at generation $T_i=2$ in the past; while the two gene copies at locus  $j$, represented by  the triangles, coalesce at generation $T_{j}=3$ in the past. The six relevant lineages are the dashed lines highlighted in yellow. The initial (present) state of the two pairs of gene copies is in state 
    $
    {
\setlength{\arraycolsep}{2pt}
q_{12}=\begin{pmatrix}
        {\color{red} \bullet }&   {\color{red} \bullet }  
        \\[-0.2cm]
         {\color{OliveGreen} \blacktriangle } &  {\color{OliveGreen} \blacktriangle }
    \end{pmatrix}
    }$.
The first step is a transition from $q_{12}$ to $q_{11}
=
{
\setlength{\arraycolsep}{3pt}\begin{pmatrix}
        {\color{red} \bullet }&   
        \\[-0.2cm]
         {\color{OliveGreen} \blacktriangle } &  
    \end{pmatrix}
\begin{pmatrix}
        {\color{red} \bullet }&   
        \\[-0.2cm]
         {\color{OliveGreen} \blacktriangle } &  
    \end{pmatrix}
    }$, 
which has probability $(1-s_N)(1-r_N)^2$.
The second step is a transition from $q_{11}$ to $q_{0}$ (coalescence occurred in at least one locus).
Note that there are other  events that can also give a transition from $q_{11}$ to $q_{0}$, such as when neither chromosome is a recombinant and both of them came from the same parental chromosome. 
    } \label{Fig:TiTj}
\end{figure}
\FloatBarrier

\section{Theoretical results}\label{S:Tresults}

Before stating our main results for two loci, we first recall some related results for a single locus for comparison.

\subsection{Coalescence time for a pair of gene copies in a single locus} \label{R:one_loci}

Fix a locus (without loss of generality, locus $i$) and consider a sample of two gene copies at that locus. As expected, the recombination probability $r_N$ does not play a role in the coalescence time for a single locus. 

All statements in this subsection, including Lemmas \ref{L:ETi}-\ref{L:Distribution_Ti}, are essentially known and summarized for instance in  \cite[ Section 2.3]{etheridge2011some}.   
However, we recall the details here because we  need these statements to set the stage for our main results in Section \ref{R:two_loci}.

The ancestral process of the sample,  obtained by tracing the ancestral lineages of the two gene copies backward-in-time, can be represented by a Markov process with 3 states $\{\text{coal},\, \text{same}, \,\text{diff}\}$ representing respectively that the gene copies have coalesced, the gene copies are in the same individual but have not coalesced, and the two gene copies are in two difference individuals.
\begin{align*}
{
\setlength{\arraycolsep}{3pt}
    \text{coal}=\begin{pmatrix}
        {\color{red} \bullet }
          &
    \end{pmatrix},\ \text{same} = \begin{pmatrix}
        {\color{red} \bullet }&  
         {\color{red} \bullet }
    \end{pmatrix},\ \text{diff} = \begin{pmatrix}
        {\color{red} \bullet }&  
    \end{pmatrix} \begin{pmatrix}
          {\color{red} \bullet }&
    \end{pmatrix}
    }
\end{align*}
From the Wright-Fisher model, the one-step transition matrix $\mathbf{\Lambda}_N$  of this Markov process is 
\FloatBarrier
\begin{figure}[ht!]
\begin{center}
\begin{tabular} {c|ccc}
  & $\text{coal}$  & $\text{same}$ & $\text{diff}$\\ \hline
  $\text{coal}$  &  $1$ & $0$ & $0$ \\[6pt]
  $\text{same}$  & $\frac{s_N}{2}$ & $\frac{s_N}{2}$ & $1-s_N$\\[6pt]
  $\text{diff}$  &  $\frac{1}{2N}$  & $\frac{1}{2N}$ &  $1-\frac{1}{N}$
\end{tabular}
\caption{Transition matrix $\mathbf{\Lambda}_N$ for a pair of gene copies at a single locus, under the Wright-Fisher model with partial selfing. The coalescent state $``\text{coal}"$ is an absorbing state.}\label{M:single_locus}
\end{center}
\end{figure}
\FloatBarrier

This matrix appeared in \cite[Section 3.1]{Mohle1998a} and \cite[eqn.(6.13)]{wakeley2009coalescent}. 
In the following, for $q\in \{\text{coal}, \text{same}, \text{diff}\}$ we let  $\mathbb P_q$  be the conditional probability under which the initial state (sampling  configurations) is $q$.

By definition, $T_i=0$  under $\P_\text{coal}$.
Through a first step analysis, we obtain explicit formulas for the first two moments of $T_i$ under both  $\mathbb P_\text{same}$ and $\mathbb P_\text{diff}$.

\begin{lemma}\label{L:ETi}\rm
For each positive integer $N\ge 2$,
\begin{enumerate}
    \item $\mathbb E_\text{same}[T_i]=2(1-s_N)\,N+2$
    \item $\mathbb E_\text{diff}[T_i]=(2-s_N)\,N + 1$
    \item  $\mathbb E_\text{same}[T_i^2] = 4(1-s_N)(2-s_N)\,N^2+10(1-s_N)\,N+6 $
    \item $\mathbb E_\text{diff}[T_i^2] = 2(2-s_N)^2\,N^2 +(6-5s_N)\,N+3$
\end{enumerate}
Thus,
\begin{enumerate}
    \item $\textrm{Var}_\text{same}[T_i] =  4(1-s_N )N^2 +2(1-s_N )N +2$
    \item $\textrm{Var}_\text{diff}[T_i] = (2-s_N )^2\,N^2 +(2-3s_N )N +2$
\end{enumerate}
\end{lemma}

Next, in Lemma \ref{L:Distribution_Ti}  we consider the asymptotic value of $T_i$ as $N\to\infty$. Precisely, we show that if one unit of  time is $N$ generations, then $T_i$ (measured in the new time unit) converges in distribution as  $N\to\infty$. 
We consider two cases for the selfing probability $s_N$: namely, either   $s_N=s$ is a constant for all $N$, or $s_N$ is inversely proportional to  $N$ (in which case we suppose $s_N=\tilde{\sigma}/N$ for all $N$). 

\begin{lemma}\label{L:Distribution_Ti}\rm
Let $s\in [0,1)$ and $\tilde{\sigma}\in [0, \infty)$. Then the following hold for all $t\in (0, \infty)$.
\begin{enumerate}
    \item When the sampled gene copies are in the same individual in generation $g=0$,
\begin{equation}
      \lim_{N\rightarrow\infty}\mathbb P_{\rm same}(T_i>[Nt]) =
      \begin{cases}
         e^{-t/(2-s)} \frac{2(1-s)}{2-s}&\text{ when } s_N=s\label{R1:P_s_tail}\\
          e^{-t/2} &\text{ when }s_N=\tilde{\sigma}/N
    \end{cases}.
\end{equation}
\item When the sampled gene copies are in distinct individuals in generation $g=0$,
\begin{equation}
      \lim_{N\rightarrow\infty}\mathbb P_{\rm diff}(T_i>[Nt]) 
     =
      \begin{cases}
          e^{-t/(2-s)}&\text{ when } s_N=s\label{R1:P_d_tail}
         \\
          e^{-t/2} &\text{ when }s_N=\tilde{\sigma}/N
    \end{cases}.
\end{equation}
\end{enumerate}
\end{lemma}

It follows from \eqref{R1:P_d_tail} that, under $\P_{\rm diff}$, $T_i/N$ converges to an exponential random variable with mean $2-s$ (resp. $2$) if $s_N=s$ (resp. $s_N=\tilde{\sigma}/N$) for all $N$.
Similarly,
from \eqref{R1:P_s_tail}, under $\P_{\rm same}$, $T_i/N$ converges to
a random variable $T_{\ast}$ as $N\to\infty$. This agrees with the standard coalescent model when $s=0$ and $N$ is itself the effective population size $N_e$.

When $s_N=O(1/N)$, this random variable $T_{\ast}$ is
an exponential random variable with mean  $2$, which is the standard coalescent. 
When $s_N=s\in(0,1)$ for all $N$, however, then 
$$\P(T_*> 0)=\frac{2(1-s)}{2-s}<1 \quad\text{and hence}\quad  \P(T_*= 0)=\frac{s}{2-s}>0.$$ 
The positive probability $\P(T_*= 0)$ of instantaneous coalesce corresponds to the well-known increase in the frequency of homozygotes in partial selfers \citep{nordborg1997coalescent}.
This probability $\frac{s}{2-s}=\sum_{i=1}^\infty(\frac{s}{2})^i$  is exactly the probability the sampled gene copies coalesce as a result of selfing rather than ending up in different individuals by outcrossing. Equivalently,
 $1-\frac{s}{2-s} = \frac{2(1-s)}{2-s}$ is
the probability that the loci end up in different individuals before they coalesce. 

The conditional distribution of $T_*$ given the event  $\{T_*>0\}$ is exponential with mean $2-s$ (same as the distribution of gene copies sampled from distinct individuals, as anticipated).
In the case $s_N=s\in(0,1)$ for all $N$, as discussed in \cite{nordborg2000linkage}, coalescence with partial selfing looks like the standard coalescent but with a rate of coalescent that is faster.

Following \cite{pollak1987theory}, we say that the \textit{effective population number} (of diploid individuals) is $\frac{2-s}{2}=1-\frac{s}{2}$ (see also \cite{EWENS1982373}). When partial selfing and recombination are both present, there maynot be a single effective population size, as pointed out in \citep{nordborg2000linkage}.







\subsection{Correlation of coalescence times for two loci} \label{R:two_loci}

As mentioned in Section \ref{S:ancrestralProcess},
the ancestry of two pairs of gene copies at two  loci can be represented by a discrete-time Markov Chain with  13 states $\{q_i\}_{i=0}^{12}$, listed  in \ref{A:States},  that represent
all possible positions of the two pairs of two gene copies. The coalescent state is considered to be the state where at least one pair of gene copies has coalesced, and is denoted by $q_0$.
The Markov transition matrix is also provided in \ref{A:States}.


Because of their biological relevance, here we focus on
the following three initial  sampling  configurations:
 {
\setlength{\arraycolsep}{3pt}
\begin{align}\label{sampling}
    q_{5}=
    \begin{pmatrix}
        {\color{red} \bullet }&  {\color{red} \bullet }
        \\[-0.2cm]
          &  
    \end{pmatrix}
    \begin{pmatrix}
         &   
        \\[-0.2cm]
         {\color{OliveGreen} \blacktriangle } & {\small \color{OliveGreen} \blacktriangle }
    \end{pmatrix},
\qquad 
q_{11}=\begin{pmatrix}
        {\color{red} \bullet }&   
        \\[-0.2cm]
         {\color{OliveGreen} \blacktriangle } &  
    \end{pmatrix}
\begin{pmatrix}
        {\color{red} \bullet }&   
        \\[-0.2cm]
         {\color{OliveGreen} \blacktriangle } &  
    \end{pmatrix}
\quad \text{or}\quad
    q_{12}=
 \begin{pmatrix}
        {\color{red} \bullet }&  {\color{red} \bullet }
        \\[-0.2cm]
         {\color{OliveGreen} \blacktriangle } & {\color{OliveGreen} \blacktriangle }
\end{pmatrix}
\end{align}
}
Namely, $q_{12}$ corresponds to picking one individual and sampling two pairs of gene copies present on both   chromosomes, $q_{11}$ corresponds to picking two individuals and then picking one chromosomes from each, $q_{5}$ corresponds to picking two individuals and then looking at one pair of gene copies, at the same locus, from each individual. 

Our main result is the asymptotic behaviors of the covariance $\text{Cov}_{q}[T_i, T_j]$ and the  correlation coefficient $\text{Corr}_{q}[T_i, T_j]$ as $N\to\infty$, where $q\in\{q_5, q_{11}, q_{12}\}$.
Corresponding results for other  initial sampling  configurations can be found in \ref{A:expectation} and \ref{A:smallcov}.


Let $R=(1-r_N)^2+r_N^2=2r_N^2-2r_N+1$, which is the probability that both chromosomes undergo recombination or both chromosomes do not undergo recombination. 
For any sequences $\{A_N\}$ and $\{B_N\}$, we write $A_N\asymp B_N$ if $\lim_{N\to\infty}\frac{A_N}{B_N}=1$ and say that $A_N$ and $B_N$ have the same order in $N$. 
\begin{thm}\label{T:q12}
Let $s\in[0,1)$ and $r\in(0, 1]$ and $\tilde{\sigma},\tilde{\rho}\in(0,\infty)$. The following asymptotics for the covariance $\text{Cov}_{q}[T_i, T_j]$ hold as $N\to\infty$.
\vspace{-1em}
\begin{align}
\intertext{\item[(i)] When $s_N=\tilde{\sigma}/N$ and  $r_N=\tilde{\rho}/N$ for all $N\in\mathbb{N}$, }
\text{Cov}_{q}[T_i, T_j] \asymp 
&\begin{dcases}
    \frac{8}{8\tilde{\rho}^2+26\tilde{\rho}+9}\cdot N^2, \quad\text{if } q=q_{5}\\
    \frac{8\tilde{\rho}+36}{8\tilde{\rho}^2+26\tilde{\rho}+9}\cdot N^2, \quad\text{if } q=q_{11}\\
    \frac{8\tilde{\rho}+36}{8\tilde{\rho}^2+26\tilde{\rho}+9}\cdot N^2, \quad\text{if } q=q_{12}
\end{dcases}\label{cov1}\\
\intertext{\item[(ii)] When $s_N=s$ and  $r_N=\tilde{\rho}/N$ for all $N\in\mathbb{N}$,}
\text{Cov}_{q}[T_i, T_j] \asymp 
&\begin{dcases} \frac{8(1-s)^2}{8\tilde{\rho}^2(1-s)^2+26\tilde{\rho}(1-s)+9}\cdot N^2, &\text{if } q=q_{5}\\
\frac{(2-s)^2(2\tilde{\rho}(1-s)
+9)}{8\tilde{\rho}^2(1-s)^2+26\tilde{\rho}(1-s)+9}\cdot N^2, &\text{if } q=q_{11}
\\
\frac{4(1-s)(4\tilde{\rho}^2s(1-s)^2+12\tilde{\rho}s(1-s)+2\tilde{\rho}(1-s)+9)}{8\tilde{\rho}^2(1-s)^2+26\tilde{\rho}(1-s)+9}\cdot N^2, &\text{if } q=q_{12}
\end{dcases}\label{cov2}\
\
\intertext{\item[(iii)] When $s_N= \tilde{\sigma}/N$ and  $r_N=r$ for all $N\in\mathbb{N}$,}
\text{Cov}_{q}[T_i, T_j] \asymp
&\begin{dcases}
   O(1), &\text{if } q=q_{5}\\
   \frac{2R}{r(2-r)}\cdot N, &\text{if } q=q_{11}\\
   \frac{2R\cdot (2r\tilde{\sigma}-2r-r^2\tilde{\sigma}+r^2+1)}{r(2-r)}\cdot N, &\text{if } q=q_{12} \\
\end{dcases}\label{cov3}\\
\intertext{\item[(iv)] When $s_N=s$ and  $r_N=r$ for all $N\in \mathbb{N}$,}
\text{Cov}_{q}[T_i, T_j]\asymp &\begin{dcases}
\frac{2s^3(1-s)(2R-2s R+s)}{(4-s^2)(2-s R)}\cdot N, &\text{if } q=q_{5}\\
\frac{\left(2-s\right)\;P(r,s)}{4r\left(1-s\right)\left(s R-2\right)\left(rs+r-
s-2\right)\left(2rs-s+2\right)}\cdot N, 
\quad &\text{if }q=q_{11}\\
\left(\frac{2(1-s)(2-s)^2}{2-s R}-4(1-s)^2\right)\cdot N^2, \quad & \text{if }q=q_{12}
\end{dcases}\label{cov4}
\end{align}
where $R=(1-r)^2+r^2$ 
and $P(r,s)$ \label{E:complexpolynomial} is the polynomial
\begin{align*}
P(r,s):=&\,(R-1)(2r-1)^3s^5-(R+2r^2)(2R-1)s^4+(32r^4-40Rr+8r+4)s^3\\
&\,+8r\left(2r-1\right)s^2-\left(40R-24\right)s+16R.
\end{align*}
\end{thm}
\medskip


\medskip

Theorem \ref{T:q12_extreme} below covers the cases of total selfing and no recombination, which are not covered in Theorem \ref{T:q12}. Note that we have exact equality for some of the terms when $N\geq 2$. This is denoted by $=$ rather than $\asymp$.
\begin{thm}\label{T:q12_extreme}
The followings hold for all integers $N\geq 2$.
\begin{align}
\intertext{(Total Selfing) When $s_N=1$}
&\begin{dcases}
    \text{Cov}_{q_{5}}[T_i, T_j]=\frac{2(2r_N-1)^2}{(3N+1)(1+2r_N-2r_N^2)}\quad &\text{if }q=q_{5}\\
    \text{Cov}_{q_{11}}[T_i, T_j]\asymp N^2\quad &\text{if }q=q_{11}\\
    \text{Cov}_{q_{12}}[T_i, T_j]=\frac{6}{1+2r_N-2r_N^2}-4 \quad &\text{if }q=q_{12}
\end{dcases}
\intertext{(No recombination) When $r_N=0$:}
&\begin{dcases}
\text{Cov}_{q_5}[T_i, T_j]\asymp 8(1-s_N)^2/9\cdot N^2 \quad &\text{if }q=q_{5} \\
\text{Cov}_{q_{11}}[T_i, T_j]=(4-4s_N +s_N^2 )N^2 +(2-3s_N )N +2 \quad &\text{if }q=q_{11}\\
\text{Cov}_{q_{12}}[T_i, T_j]= (4-4s_N )N^2 +(2-2s_N )N +2 \quad &\text{if }q=q_{12}
\end{dcases}
\end{align}
\end{thm}

The proofs of Theorem \ref{T:q12} and Corollary \ref{C:q12q11} will be provided in \ref{A:Proofs}. Note that the proof also gives similar results for different  initial sampling  configurations i.e. for all states $q\in\mathcal S$. 

\medskip

\section{Remarks about our analytical formulas}\label{S:interpretation}

We now offer some explanation of the analytical results in section \ref{R:two_loci} to provide some feeling for those complicated formulas. 
Since it is also of interest to look at the correlation coefficient which is normalized and therefore unitless, we first rephrase Theorems \ref{T:q12} and \ref{T:q12_extreme} as results about the correlation coefficient.
\begin{corollary}\label{C:q12q11}
Let $s\in[0,1)$ and $r\in(0, 1]$ and $\tilde{\sigma},\tilde{\rho}\in(0,\infty)$. The following asymptotics for the correlation coefficient $\text{Corr}_{q}[T_i, T_j]$ hold as $N\to\infty$.
\begin{align}
\intertext{\item[(i)] When $s_N=\tilde{\sigma}/N$ and  $r_N=\tilde{\rho}/N$ for all $N\in\mathbb{N}$,} 
\text{Corr}_{q}[T_i, T_j] \asymp
&\begin{dcases}
    \frac{2}{8\tilde{\rho}^2+26\tilde{\rho}+9}, \quad &\text{if } q=q_{5}\\
    \frac{2\tilde{\rho}+9}{8\tilde{\rho}^2+26\tilde{\rho}+9}, \quad &\text{if } q=q_{11}\\
    \frac{2\tilde{\rho}+9}{8\tilde{\rho}^2+26\tilde{\rho}+9}, \quad &\text{if } q=q_{12}
\end{dcases} \label{E:coeff1}\\
\intertext{\item[(ii)] When $s_N=s$ and  $r_N=\tilde{\rho}/N$ for all $N\in\mathbb{N}$,}
\text{Corr}_{q}[T_i, T_j] \asymp 
&\begin{dcases}
    \frac{2}{8\tilde{\rho}^2(1-s)^2+26\tilde{\rho}(1-s)+9}, &\text{if } q=q_{5}\\
    \frac{2\tilde{\rho}(1-s)+9}{8\tilde{\rho}^2(1-s)^2+26\tilde{\rho}(1-s)+9}&\text{if } q=q_{11},  \\
    \frac{4\tilde{\rho}^2s(1-s)^2+12\tilde{\rho}s(1-s)+2\tilde{\rho}(1-s)+9}{8\tilde{\rho}^2(1-s)^2+26\tilde{\rho}(1-s)+9}&\text{if } q=q_{12}, 
\end{dcases} \label{E:coeff2}\\
\intertext{\item[(iii)] When $s_N= \tilde{\sigma}/N$ and  $r_N=r$ for all $N\in\mathbb{N}$,}   
\text{Corr}_{q}[T_i, T_j] \asymp 
&\begin{dcases}
    O(N^{-2}), \quad &\text{if } q=q_{5}\\
    \frac{R}{2r(2-r)}\cdot \frac{1}{N}, &\text{if } q=q_{11}\\
    \frac{R\cdot (2r\tilde{\sigma}-2r-r^2\tilde{\sigma}+r^2+1)}{2r(2-r)}\cdot \frac{1}{N}, &\text{if } q=q_{12}
\end{dcases} \label{E:coeff3}
\intertext{
\item[(iv)] When $s_N=s$ and  $r_N=r$ for all $N\in \mathbb{N}$, } 
\text{Corr}_{q}[T_i, T_j] \asymp &\begin{dcases}
\frac{s^3(2R-2s R+s)}{2(4-s^2)(2-s R)}\cdot \frac{1}{N}, &\text{if } q=q_{5}\\
\frac{P(r, s)}{4r\left(2-s\right)\left(1-s\right)\left(s R-2\right)\left(rs+r-
s-2\right)\left(2rs-s+2\right)} \cdot \frac{1}{N}\;\; &\text{if }q=q_{11}\\
\frac{(2-s)^2}{2(2-s R)}+s-1, \quad &\text{if }q=q_{12}
\end{dcases}
\label{E:coeff4}
\end{align}
where $R=(1-r)^2+r^2$ and $P(r, s)$ is the polynomial defined in Theorem \ref{T:q12}. 
\end{corollary}

\begin{corollary}\label{C:q12_extreme}
The following statements hold for all integers $N\geq 2$.
\begin{align}
\intertext{(Total Selfing) When $s_N=1$}
&\begin{dcases}
    \text{Corr}_{q_5}[T_i, T_j]=\frac{(2r_N-1)^2}{(3N+1)(1+2r_N-2r_N^2)}\quad &\text{if }q=q_{5}\\
    \text{Corr}_{q_{11}}[T_i, T_j]\asymp 1\quad &\text{if }q=q_{11}\\
    \text{Corr}_{q_{12}}[T_i, T_j]=\frac{3}{1+2r_N-2r_N^2}-2\quad &\text{if }q=q_{12}
\end{dcases}
\intertext{(No recombination) When   $r_N=0$.}
&\begin{dcases}
\text{Corr}_{q_5}[T_i, T_j]\asymp 2(1-s_N)/9\quad &\text{if }q=q_{5}\\  \text{Corr}_{q_{11}}[T_i, T_j]=1\quad &\text{if }q=q_{11}\\    \text{Corr}_{q_{12}}[T_i, T_j]=1 \quad &\text{if }q=q_{12}
\end{dcases}
\end{align}
\end{corollary}

\bigskip


Notice that the correlation coefficient is of order 1 when $r_N=O(1/N)$, regardless of whether $s_N$ is of order $O(1)$ or $O(1/N)$.
This makes sense since the correlation is large (or order $O(1)$) when recombination probability is small. 
In the extreme case when $\tilde{\rho}\downarrow 0$ in scenarios (i) and (ii), we have $\text{Corr}_{q}[T_i, T_j]\to 1$ for $q\in\{q_{11},q_{12}\}$.
Furthermore, when $r_N=O(1)$ (high recombination rate), the correlation coefficients are small (of order $O(1/N)$ or smaller) except when $q=q_{12}$ and $s=O(1)$ simultaneously.

\subsection{Scenario (i), \texorpdfstring{$s_N=\tilde{\sigma}/N$ and $r_N=\tilde{\rho}/N$}{sigma\_N=sigma'/N and rN='/N}} 
If we take the population scaled recombination rate to be $\rho = 4Nr_N$ (i.e. $\tilde{\rho}=\rho/4$), then
the RHS of \eqref{E:coeff1} becomes
$$\frac{2\tilde{\rho}+9}{8\tilde{\rho}^2+26\tilde{\rho}+9} =\frac{\rho+18} {\rho^2+13\rho+18} \qquad \text{for any initial state in }\{q_{11},\,q_{12}\}$$ 
which agrees exactly with \citet{griffiths1981neutral}, \citet[eqn. (A13)]{pluzhnikov1996optimal}, \citet[eqn.(5)]{kaplan1985use}, \citet[eqn.(13)]{hudson1990gene}, and \citet[eqn. (A35)]{Birkner255} if $\tilde{\rho}=2r$. See also  \citet[(7.28)-(7.30)]{wakeley2009coalescent}. Given that the selfing probability is $O(1/N)$ it makes sense that the result is the same regardless of whether the two chromosomes begin in the same individual or not. 

Similarly, from the results in \ref{A:expectation}, if we take  $\rho = 4Nr_N$ (i.e. $\tilde{\rho}=\rho/4$), then the limiting correlation coefficient
 \eqref{E:coeff1}  is equal to
\begin{align}
\frac{2}{8\tilde{\rho}^2+26\tilde{\rho}+9}=\frac{4} {\rho^2+13\rho+18} \qquad \text{for any initial state in }\{q_i\}_{i=1}^{6}  \\
\frac{3}{8\tilde{\rho}^2+26\tilde{\rho}+9}=\frac{6} {\rho^2+13\rho+18} \qquad \text{for any initial state in }\{q_i\}_{i=7}^{10}
\end{align}
which agree with \cite[(7.29)-(7.30)]{wakeley2009coalescent}.

\subsection{Scenario (ii), \texorpdfstring{$s_N=s$ and $r_N=\tilde{\rho}/N$}{sN=s and rN='/N}}

If we plug $s=0$ into \eqref{E:coeff2} we obtain the exact formulas in \eqref{E:coeff1}. This makes sense because $\tilde{\sigma}/N\to 0$  when $N\to\infty$. This scenario is what
 \cite{nordborg2000linkage} considered in their setting.

For state $q_{12}$, as we increase the selfing probability to $1$, $\textrm{Corr}_q[T_i ,T_j]$ approaches $1$. This makes sense since the loci will stay in the same individual and undergo recombination with low probability. Hence $\mathbb P(T_i=T_j)\to 1$.

An interesting point is that the correlation coefficient stays positive even when $\tilde{\rho}\to\infty$. Precisely,
\begin{equation}\label{positiveCorr_ii}
    \lim_{\tilde{\rho}\to\infty} \lim_{N\to\infty}\text{Corr}_{q_{12}}[T_i, T_j]=\frac{s}{2} >0.
\end{equation}

\subsection{Scenario (iii), \texorpdfstring{$s_N=\tilde{\sigma}/N$ and $r_N=r$}{sigma\_N=sigma'/N and r\_N=r}}
\begin{itemize} \item \textbf{State $q_{12}$:} If we set $\tilde{\sigma}=1$, or $s_N=1/N$ we can compare the results to the Simplified DDTWF from \cite{king2018non}. Even though \cite{king2018non} uses a two-sex model, the simplified model does not keep track of whether lineages are in the same individual or not. We note that if $\tilde{\sigma}=1$ then the formula for $\text{Corr}_{q_{12}}[T_i, T_j]$ becomes exactly $\frac{2r^2-2r+1}{2r(2-r)}\cdot\frac{1}{N}$, which is equal to $\text{Corr}_{q_{11}}[T_i, T_j]$ in \eqref{E:coeff3}. This makes sense since when $\tilde{\sigma}=1$ we obtain a uniform distribution for parent chromosomes meaning we expect $\text{Corr}_{q_{11}}[T_i, T_j]=\text{Corr}_{q_{12}}[T_i, T_j]$. Additionally, setting $r=1/2$ gives that  
\begin{equation}\label{GeneralizedKingetal}
\text{Corr}_{q_{12}}[T_i ,T_j] \asymp \frac{3\tilde{\sigma}+1}{12N}.    
\end{equation}
Note when $\tilde{\sigma}=0$ (i.e. Wright-Fisher with no selfing similar to two-sex model) we also obtain $\textrm{Corr}_{q_{12}}[T_i, T_j]\asymp \frac{1}{12N}$ which is the same as the asymptotic  lower bound in \cite{king2018non}. 

\item \textbf{State $q_{11}$:} When $r=1/2$, the RHS of \eqref{E:coeff3} gives that  $\text{Corr}_{q_{11}}[T_i ,T_j] \asymp \frac{1}{3N}$. Here $q_{11}$ corresponds to to the simplified DDTWF Config 1, in \cite{king2018non}.

\item \textbf{State $q_{5}$:} Here we obtain an expression for $\textrm{Corr}_{q_{5}}[T_i, T_j]$ that is of order $O(N^{-2})$. In states $q_{11}$ and $q_{12}$, the correlation was $O(N^{-1})$ since there is an $O(N^{-1})$ probability that all four gene copies belong to the same individual before coalescence. This is not the case for state $q_5$, where  the two pairs (of gene copies) begin in two different individuals, and there is $O(N^{-2})$ probability that all gene copies belong to the same individual before a coalescence occur.
\end{itemize}

\FloatBarrier

\subsection{Scenario \texorpdfstring{$(iv)$, $s_N=s$ and $r_N=r$}{TEXT}}

\begin{itemize}
    \item \textbf{State $q_{12}$:} First, we point out an analogous statement to \eqref{positiveCorr_ii}. Namely, the asymptotic correlation coefficients stays positive even when the recombination probability tends to 1:
\begin{equation}\label{positiveCorr_iv}
    \lim_{r\to 1} \lim_{N\to\infty}\text{Corr}_{q_{12}}[T_i, T_j]=\frac{s}{2} >0.
\end{equation}

Next,  we explain why the complicated expression in \eqref{cov4}  make sense by giving an alternative, simpler argument. Recall that we let $R=2r^2-2r+1$. Then $-2r^2+2r+1=2-R$. We shall show that
    \begin{equation}\label{q12_simple}
    \E_{q_{12}}[T_iT_j] \asymp \frac{2(1-s)(2-s)^2}{(2-s R)}\cdot N^2    \qquad \text{ as }   N\to\infty. 
    \end{equation}

\begin{proof}[Proof of \eqref{q12_simple}]
We will solve for the asymptotics of $\E_{q_{12}}[T_iT_j]$ through a simplified first step analysis of our 13$\times$13 transition matrix $\bf \Pi_N$ given in 
\ref{A:Matrix}

We note from \ref{A:Matrix} that the transition probabilities from $q_{12}$ to $q_0$ and to $q_{12}$ are  $\frac{s(2-R)}{2}$ and $\frac{s R}{2}$ respectively, independent of $N$. 
This implies that the transition probability from $q_{12}$ to any state other than $q_0$ or $q_{12}$ is $1-\frac{s(2-R)}{2}-\frac{s R}{2}=1-s$.  
                
If we transition to $q_0$ (in the first step) then the value of $\E [T_iT_j]$ is $O(N)$ since either $T_i=1$ or $T_j=1$, and $\E [T_i]$ and $\E[T_j]$ are both $O(N)$. 
With probability $1-s$ we  transition to one of the states in $\{q_1, \dots, q_{11}\}$. According to tables \eqref{E:limN^2} and \eqref{E:limNN}, $\lim_{N\to\infty}\E_{q}[T_iT_j]/N^2=\lim_{N\to\infty}\E_{q}[T_i]\E_{q}[T_j]/N^2$ whenever $q\in\{q_1, \dots, q_{11}\}$ in Scenario $(iv)$. 

Furthermore, conditional on  transitioning to one of the states in $\{q_1, \dots, q_{11}\}$, we must end up in one of the states $\{q_6, q_{10}, q_{11}\}$. It is impossible to transition from $q_{12}$ to any of the states $q_1$ through $q_9$ since loci from the same chromosome must end up in the same parent individual. For any state $q$ in which the pairs of gene copies at locus $i$ and locus $j$ both  belong to different individuals, we have
$(2-s)^2=\lim_{N\to\infty}\E_{q}[T_i]\E_{q}[T_j]/N^2=\lim_{N\to\infty}\E_{q}[T_iT_j]/N^2$  according to Lemma \ref{L:ETi}. This is the case for states $q_6, q_{10}, q_{11}$. 

In summary, by  \eqref{E:expectation}, 
    \begin{align} 
    &\E_{q_{12}}[T_iT_j] \\
    =&\, \sum_{k\in \mathcal S} ({\bf \Pi}_N)_{q_{12},k}\,\E_k[T_iT_j]+\E_{q_{12}}[T_i]+\E_{q_{12}}[T_j]-1\\
     \asymp&\, ({\bf \Pi}_N)_{q_{12},q_{12}}\E_{q_{12}}[T_iT_j]+\sum_{k\in \mathcal S-\{q_0, q_{12}\}}({\bf \Pi}_N)_{q_{12}, k} (2-s)^2\cdot N^2 \\
    \asymp&\, \frac{s R}{2}\cdot \E_{q_{12}}[T_iT_j] + (2-s)^2\cdot N^2,
    \end{align}   
 where  we discard terms that are of order $O(N)$ or smaller, and only keep the terms that are or order  $O(N^2)$ on the right hand side.  Upon grouping the $\E_{q_{12}}[T_iT_j]$ terms gives
 \eqref{q12_simple}.
\end{proof}

From the previous Lemma we immediately obtain \eqref{cov4} for $q_{12}$.

    \item \textbf{State $q_{11}$:} For state $q_{11}$ plugging in $s=0$ into \eqref{E:coeff4} we obtain the exact formulas for $\text{Corr}_{q_{11}}[T_i, T_j]$ in \eqref{E:coeff3}. This makes sense because $\tilde{\sigma}/N\to 0$ when $N\to\infty$. 
    \item \textbf{State $q_{5}$:} For state $q_5$ plugging in $s=0$ yields $\text{Corr}_{q_{5}}[T_i, T_j]\asymp 0$, which is as anticipated since in \eqref{E:coeff3} the expression for $\text{Corr}_{q_5}[T_i, T_j]$ is of order $O(1/N^2)$.
\end{itemize}

\FloatBarrier


\subsection{Variance of Tajima's estimator}

An application of Theorem \ref{T:q12} is that it immediately give asymptotic formulas for the variance of Tajima's estimator for all 4 scenarios for the strengths of recombination and of selfing.
To see this, let $\mu$ be the mutation rate per locus per generation.
The Tajima's estimator $\hat\theta_{(n)}$ for the population-scaled mutation rate $\theta=4N\mu$, for a sample of $n$ pairs of loci, is given by 
$$\hat \theta_{(n)}:=\frac{1}{n}\sum_{i=1}^n\hat \theta_i,$$ 
where $\hat \theta_i$ is the average number of pairwise differences for the sample at locus $i$. 
From \cite[eqn. (4)]{king2018non},  under the assumption of the infinite site model,
$$\lim_{n\to\infty}\text{Var}_{q}[\hat\theta_{(n)}] = 4\mu^2\text{Cov}_{q}[T_i, T_j]$$
for each fixed population size $N$, where
the asymptotic value $\text{Cov}_{q}[T_i, T_j]$ as $N\to\infty$ can be read off from Theorem \ref{T:q12}.

For example,  the right hand side 
 of the above display for sampling configuration $q_{12}$ is
\[
 4\mu^2\text{Cov}_{q_{12}}[T_i, T_j]\asymp  4\left(\frac{\theta}{4N}\right)^2\text{Corr}_{q_{12}}[T_i, T_j]\,4N^2 \asymp \theta^2\frac{3\tilde{\sigma}+1}{12N}  
\]
when $s_N=\tilde{\sigma}/N$ and $r_N=1/2$ for all $N$ (i.e. unlinked loci in Scenario (iii)). This follows from Lemma \ref{L:ETi} and \eqref{GeneralizedKingetal}. Similar asymptotics for $\lim_{n\to\infty}\text{Var}[\hat\theta_{(n)}]$ can be obtained for other scenarios by using Theorem \ref{T:q12}.

From \eqref{positiveCorr_ii} and \eqref{positiveCorr_iv},  when $s_N=s$ is of order $O(1)$ and $q=q_{12}$, the variance of the Tajima's estimator remains positive even when recombination rate tends to infinity and for infinitely many loci:
\begin{equation}\label{Positive_Var}
\lim_{N\to\infty}\lim_{n\to\infty}\text{Var}_{q_{12}}[\hat\theta_{(n)}] \to \frac{\theta^2 (1-s)^2s}{2} >0
\end{equation}
as $\tilde{\rho}\to\infty$ in scenario (ii), or as 
$r\to1$ in scenario (iv). 

\section{Simulations}\label{S:simulations}

Here we present results from Monte Carlo simulations to support results from Corollary \ref{C:q12q11}. To estimate the value of $\text{Corr}[T_i, T_j]$, we run a Monte Carlo simulation that will generate values of $T_i$ and $T_j$. The code can be found in 
\ref{A:Code}.
The input parameters are $N, s_N, r_N$ which remain fixed for each run, and an initial state $q\in\{q_{i}\}_{i=1}^{12}$. 

In a single trial of our simulation, we obtain a pair of integers $(T_i,T_j)$ by simulating the transition matrix of the ancestral process.
Since we want our simulations to compute $T_i$ and $T_j$,  we need to refine the definition for the coalescent state $q_0$ and extend  our 13$\times$13 transition matrix $\bf \Pi_N$ (explicitly computed in \ref{A:Matrix}) to a larger matrix.

Following the notation in the Appendix of  \cite{king2018non},
we partition the state $q_0$ into five substates 
$$\mathcal{C}\,=\,\Big\{\{\}, \{1, 1\}, \{2, 2\}, \{(1, 1)\}, \{(2, 2)\}\Big\},$$
where 
\setlength{\arraycolsep}{3pt}
\begin{align*}
    \{\} = \begin{pmatrix}
     &
    \end{pmatrix} ,\ 
    \{1, 1\}=\begin{pmatrix}
        {\color{red} \bullet }&  \\
    \end{pmatrix} \begin{pmatrix}
        {\color{red} \bullet }&  \\
    \end{pmatrix},\ 
    \{(1, 1)\}=\begin{pmatrix}
        {\color{red} \bullet } &  {\color{red} \bullet }\\
    \end{pmatrix} \\
    \{2, 2\}=\begin{pmatrix}
        {\color{OliveGreen} \blacktriangle }&  
    \end{pmatrix} \begin{pmatrix}
          {\color{OliveGreen} \blacktriangle } &
    \end{pmatrix} ,\
    \{(2, 2)\} = \begin{pmatrix}
        {\color{OliveGreen} \blacktriangle }&  
         {\color{OliveGreen} \blacktriangle }
    \end{pmatrix}
\end{align*}  
in which $\{\}$ represents the state when both pairs of gene copies at the two loci have coalesced. 
States $\{1, 1\}$ and $\{(1, 1)\}$ correspond to the case when locus $j$ has coalesced, but locus $i$ has not. For $\{1, 1\}$ the gene copies at loci $i$ belong to different individuals, and for $\{(1, 1)\}$ the the gene copies at loci $i$ belong to  the same individual. $\{2, 2\}$ and $\{(2, 2)\}$ are defined analogously to $\{1, 1\}$ and $\{(1 ,1)\}$ respectively, except here locus $i$ has coalesced rather than locus $j$.
 
After a pair of gene copies in a loci have coalesced we no longer need to keep track of it, so we do not include coalesced pairs in these five substates. We 
let $\bf \Pi_N^*$ be the $17\times 17$ Markov transition matrix for the states $\mathcal{S}\cup \mathcal{C}$, in which we split the original state $q_0$ into its five sub-states. 

We are given an initial state $q$ at generation $g=0$. We would like to run our simulation backwards in time until the lineages at each locus have coalesced. We can simulate these lineages backwards in time by first simulating the state of the loci a single generation backwards in time, i.e. generation $g=1$. Simulating the state in generation $g=1$ can be done by selecting a random state  based on our 17$\times$17 transition matrix $\bf \Pi_N^*$. If we are in the state $q_i$ currently, then we pick a state in the generation prior by selecting a state at random with probabilities given by the $i$-th row of the transition matrix. We repeat this process until we have reached states in which \textit{both} pairs of loci have coalesced. Algorithm 1 gives the psuedo-code for running a single trial. 

\noindent

\FloatBarrier
\begin{algorithm}
\caption{Obtaining $\{T_i,T_j\}$ in a single trial} \label{A:algorithm}
\begin{algorithmic} 
\State Input: parameters $N, s_N, r_N$ and $\bf \Pi_N^*$, and the initial state $q$ 
\State $T_i=0,\ T_j=0$
\While{$q\ne \{\}$}
\State Randomly select a state $q_{\text{next}}$ in $\mathcal S\cup \mathcal C$ with probability corresponding to $\bf \Pi_N^*$ 
\If{$q\ne \{1, 1\}$ and $q\ne \{(1, 1)\}$}  
    \State $T_i = T_i+1$
\EndIf   
\If{$q\ne \{2, 2\}$ and $q\ne \{(2, 2)\}$} 
    \State $T_j = T_j+1$
\EndIf
\State $q=q_{\text{next}}$
\EndWhile
\State Return the pair $(T_i,\ T_j)$ 
\end{algorithmic}
\end{algorithm}
\FloatBarrier

This gives a single realization for $\{T_i,T_j\}$ (and hence one realization of the product $T_iT_j$). The code for these simulations can be found in \ref{A:Code}.

We perform $M$ independent trials to obtain $M$ simulated pairs of values $(T_i,T_j)$. Call these values $\{(T^{(k)}_i,T^{(k)}_j)\}_{k=1}^M$.
We estimate the correlation coefficient $\text{Corr}[T_i, T_j]$ by the Pearson correlation coefficient
\begin{align*}
    \frac{M(\sum_{k=1}^M T_i^{(k)}T_j^{(k)})-(\sum_{k=1}^M T_i^{(k)})(\sum_{k=1}^M T_j^{(k)})}{\sqrt {\left(M\cdot \sum_{k=1}^M (T_i^{(k)})^2-(\sum_{k=1}^M T_i^{(k)})^2\right) \left(M\cdot \sum_{k=1}^M (T_j^{(k)})^2-(\sum_{k=1}^M T_j^{(k)})^2\right) }}.
\end{align*}

We computed this for different values of $s_N$ and $r_N$, then compare this estimated value of $\text{Corr}[T_i, T_j]$ with the corresponding analytical formulas in Corollary \ref{C:q12q11}. Figures \ref{S:q12simulations} and \ref{S:q11simulations} show these comparisons for the initial states $q_{12}$ and $q_{11}$ respectively.


\subsection{Simulation results for states \texorpdfstring{$q_{12}$ and $q_{11}$}{q12 and q11}} \label{S:q12sims}

In Figure \ref{S:q12simulations}, we plot the correlation coefficient $\text{Corr}[T_i, T_j]$ for initial condition $q_{12}$ and for all four scenarios,  as a function of  $r, (\tilde{\rho})$ (left column) and  as a function of  $s, (\tilde{\sigma})$ (right column). 
The 4 figures on the left column  were plotted for three fixed selfing probabilities $s=0.1, 0.3, 0.9$ and $\tilde{\sigma}=1/N, 3/N, 5/N$. The right figures were plotted  for three fixed recombination probability $r=0.1, 0.3, 0.9$ (or $r=0.1, 0.3, 0.5$ when the correlation is symmetric w.r.t $r=1/2$) and $\tilde{\rho}=1/N, 3/N, 5/N$. In all figures \ref{S:q12simulationsa}--\ref{S:q12simulationsh}, the initial state is $q_{12}$. The population size is $N$ and the number of trials per point is $M$, which varies for different plots. The simulation results (represented by the circles) are obtained by taking average of $M$ simulations of the one-sex, diploid, discrete time Wright-Fisher model. The analytical formula, represented by lines, is given by corresponding formulas from Corollary \ref{C:q12q11}.

Figure \ref{S:q11simulations} is analogous to Figure \ref{S:q12simulations}, but for initial state $q_{11}$.

\medskip

Figures  \ref{S:q12simulations} and \ref{S:q11simulations} confirm the accuracy of our analytical results for the asymptotic value of $\textrm{Corr}[T_i, T_j]$ as $N\to\infty$, for initial samplings $q_{12}$ and $q_{11}$.
More precisely, the Monte Carlo approximation for the correlation coefficient $\textrm{Corr}[T_i, T_j]$, based on discrete simulations with $M$ independent trials and when $N$ and $M$ are both large,
agree well with our analytical formula in Corollary \ref{C:q12q11} for all 4 scenarios. 



As seen from the right columns of  Figures  \ref{S:q12simulations} and \ref{S:q11simulations}, selfing increases the correlation coefficient for fixed $r_N$. 
 In all plots except Figures \ref{S:q12simulationsb}, \ref{S:q11simulationsb} and \ref{S:q11simulationsd},  correlation is an increasing function of $s_N$. 
In Figures  \ref{S:q12simulationsb}, \ref{S:q11simulationsb} and \ref{S:q11simulationsd}, the correlation coefficient is constant in $s_N$, so selfing has negligible effect on the correlation coefficient unless $s_N$ is of order $O(1)$.
Comparing the right column of  Figure  \ref{S:q12simulations} with that of  Figure \ref{S:q11simulations}, we see that selfing has a stronger effect  in $q_{12}$ than in $q_{11}$, 
especially in Scenario (iv) in which both $r_N$ and $s_N$ are of order $O(1)$. 


As seen from the left columns of  Figures  \ref{S:q12simulations} and \ref{S:q11simulations}, recombination decreases the  correlation coefficient in general for fixed $s_N$, as expected. Worth-noting is that
correlation decreasing may not maynot decrease to $0$. 
For plots \ref{S:q12simulationsa} and \ref{S:q11simulationsa} for scenario $(i)$, correlation decreasing from $1$ to $0$ as $\tilde{\rho}\uparrow \infty$. However,
for scenarios $(ii)$ and $(iv)$ in Figures  \ref{S:q12simulations},
the correlation coefficient tends to $s/2$ as $r_N$ increases.  
Comparing the left column of  Figure  \ref{S:q12simulations} with that of  Figure \ref{S:q11simulations}, we see that recombination has a stronger effect (decreases the correlation faster as $r_N$ increases) in $q_{11}$ than in $q_{12}$, 
especially in Scenario (iv) in which both $r_N$ and $s_N$ are of order $O(1)$. 

For scenario $(iv)$ in Figure \ref{S:q12simulationsg} we see a function symmetric along $r_N=1/2$ in Figure \ref{S:q12simulationsg}.
This result in an increase in correlation when $r_N=r$ increases from $1/2$ to 1. An increase in correlation  when $r_N=r$ increases from $1/2$ to 1 is also observed in 
Figure \ref{S:q12simulationse}.
However, the case $r_N\geq 1/2$ is biologically irrelevant.

\FloatBarrier

\begin{figure}
     \centering
     \hspace*{\fill}
     \begin{subfigure}[b] {0.37\textwidth}
         \centering        \includegraphics[width=\textwidth]{ 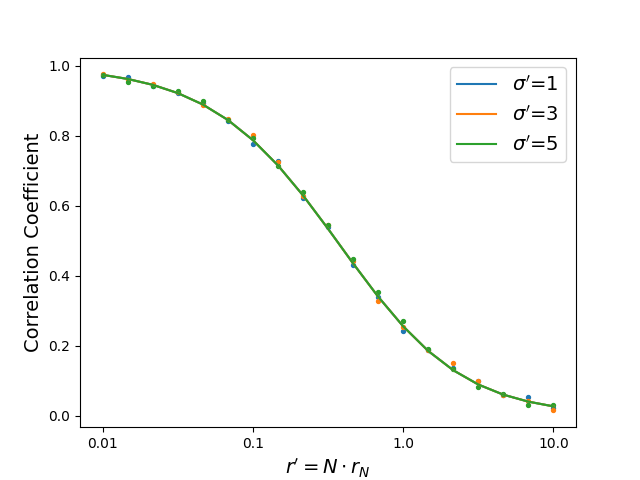}
         \caption{Scenario $(i)$ with $N=500$ and $M=20000$}
         \label{S:q12simulationsa}
     \end{subfigure}
     \hfill
     \begin{subfigure}[b]{0.37\textwidth}
         \centering
         \includegraphics[width=\textwidth]{ 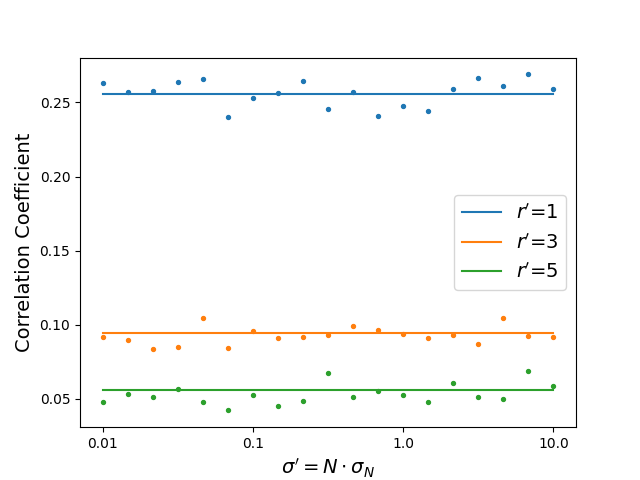}
         \caption{Scenario $(i)$ with $N=500$ and $M=20000$}
         \label{S:q12simulationsb}
     \end{subfigure}
     \hspace*{\fill}

     \hspace*{\fill}
     \begin{subfigure}[b]{0.37\textwidth}
         \centering   \includegraphics[width=\textwidth]{ 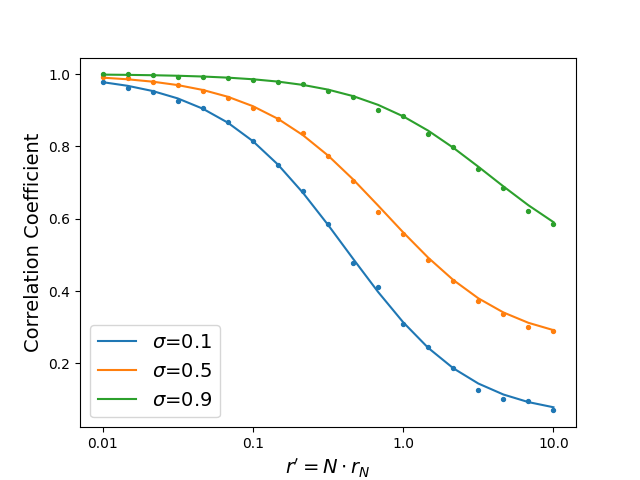}
         \caption{Scenario $(ii)$ with $N=500$ and $M=20000$}
         \label{S:q12simulationsc}
     \end{subfigure}
     \hfill
     \begin{subfigure}[b]{0.37\textwidth}
         \centering
         \includegraphics[width=\textwidth]{ 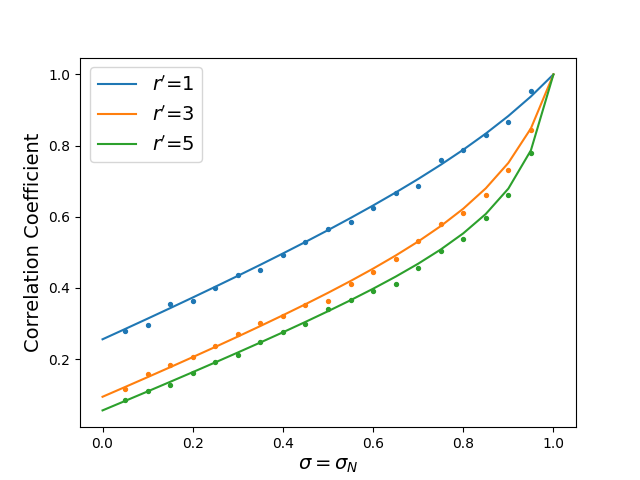}
         \caption{Scenario $(ii)$ with $N=500$ and $M=20000$}
         \label{S:q12simulationsd}
     \end{subfigure}
     \hspace*{\fill}

     \hspace*{\fill}
     \begin{subfigure}[b]{0.37\textwidth}
         \centering   \includegraphics[width=\textwidth]{ 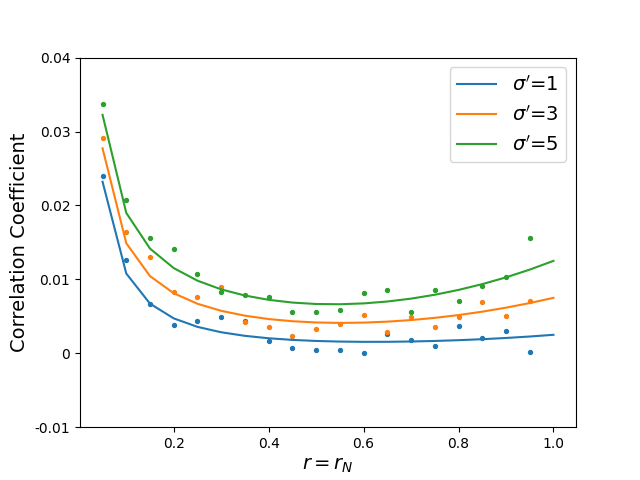}
         \caption{Scenario $(iii)$ with $N=200$ and $M=500000$}
         \label{S:q12simulationse}
     \end{subfigure}
     \hfill
     \begin{subfigure}[b]{0.37\textwidth}
         \centering         \includegraphics[width=\textwidth]{ 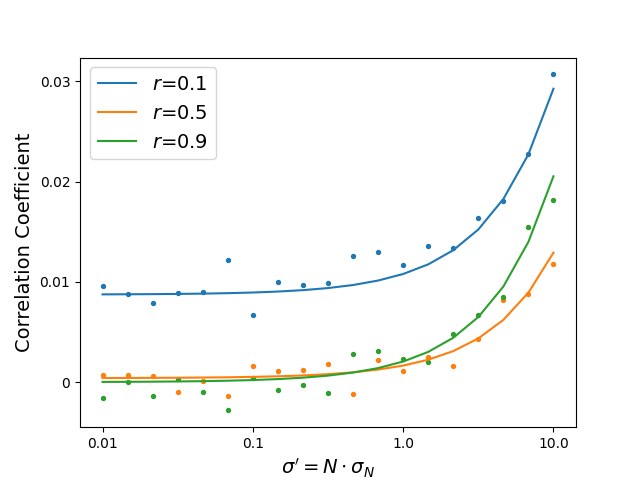}
         \caption{Scenario $(iii)$ with $N=500$ and $M=20000$}
         \label{S:q12simulationsf}
     \end{subfigure}
     \hspace*{\fill}

     \hspace*{\fill}
     \begin{subfigure}[b]{0.37\textwidth}
         \centering   \includegraphics[width=\textwidth]{ 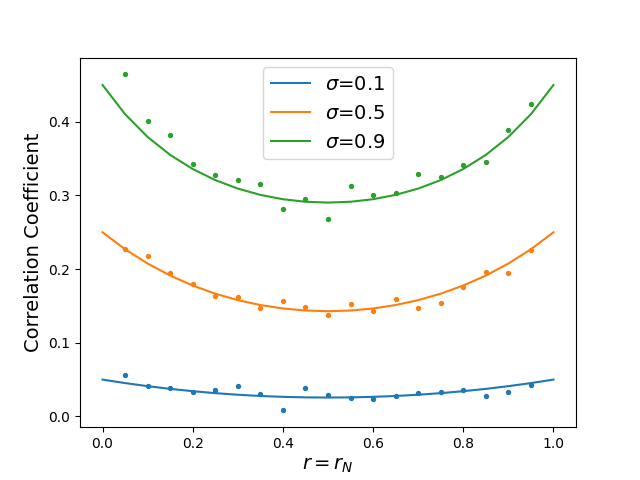}
         \caption{Scenario $(iv)$ with $N=1000, M=40000$}
         \label{S:q12simulationsg}
     \end{subfigure}
     \hfill
     \begin{subfigure}[b]{0.37\textwidth}
         \centering         \includegraphics[width=\textwidth]{ 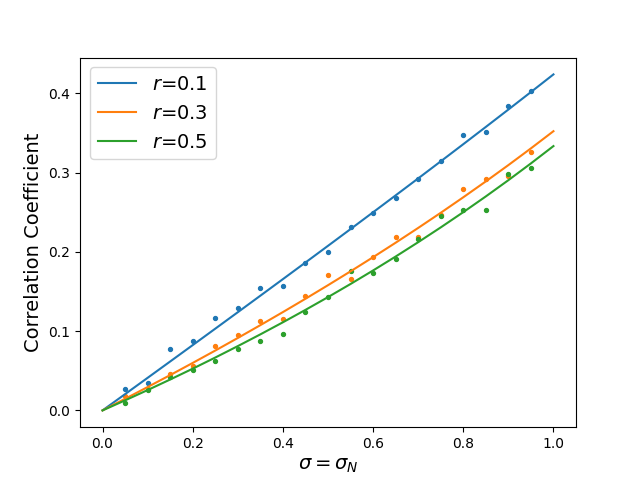}
         \caption{Scenario $(iv)$ with $N=1000, M=40000$}
        \label{S:q12simulationsh}
     \end{subfigure}
     \hspace*{\fill}
\caption{Correlation coefficient for initial state $q_{12}$} \label{S:q12simulations}
\end{figure}
\FloatBarrier



\FloatBarrier
\begin{figure}
     \centering
     \hspace*{\fill}
     \begin{subfigure}[b]{0.355\textwidth}
         \centering         \includegraphics[width=\textwidth]{ 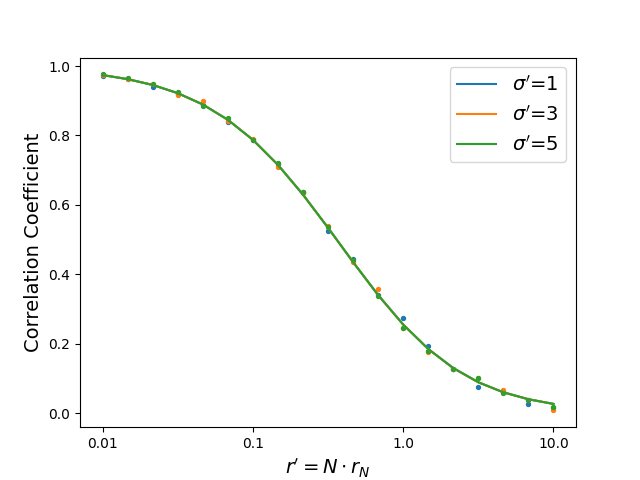}
         \caption{Scenario $(i)$ with $N=500$ and $M=20000$}
         \label{S:q11simulationsa}
     \end{subfigure}
     \hfill
     \begin{subfigure}[b]{0.355\textwidth}
         \centering         \includegraphics[width=\textwidth]{ 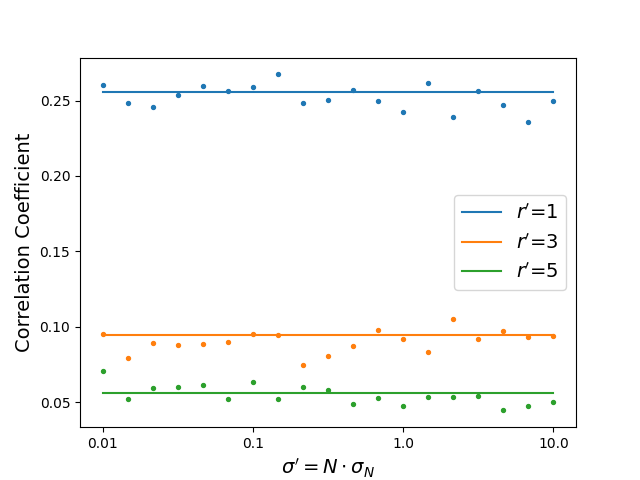}
         \caption{Scenario $(i)$ with $N=500$ and $M=20000$}
         \label{S:q11simulationsb}
     \end{subfigure}
     \hspace*{\fill}

     \hspace*{\fill}
     \begin{subfigure}[b]{0.355\textwidth}
         \centering   \includegraphics[width=\textwidth]{ 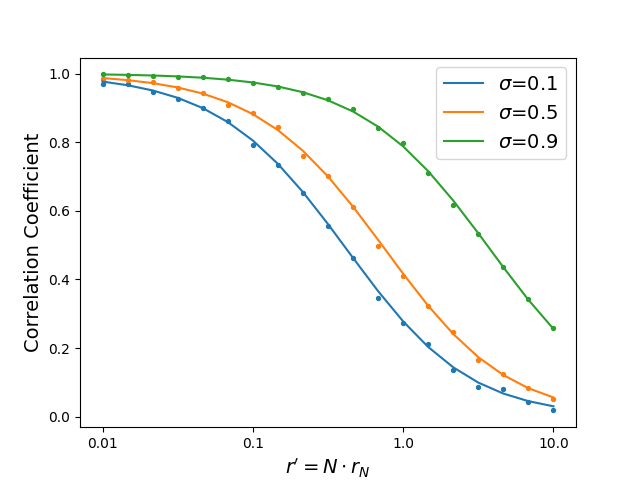}
         \caption{Scenario $(ii)$ with $N=500$ and $M=20000$}
         \label{S:q11simulationsc}
     \end{subfigure}
     \hfill
     \begin{subfigure}[b]{0.355\textwidth}
         \centering         \includegraphics[width=\textwidth]{ 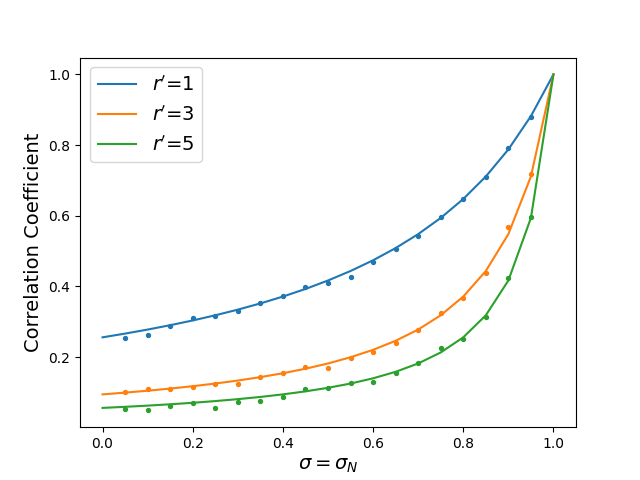}
         \caption{Scenario $(ii)$ with $N=500$ and $M=20000$}
         \label{S:q11simulationsd}
     \end{subfigure}
     \hspace*{\fill}

     \hspace*{\fill}
     \begin{subfigure}[b]{0.355\textwidth}
         \centering   \includegraphics[width=\textwidth]{ 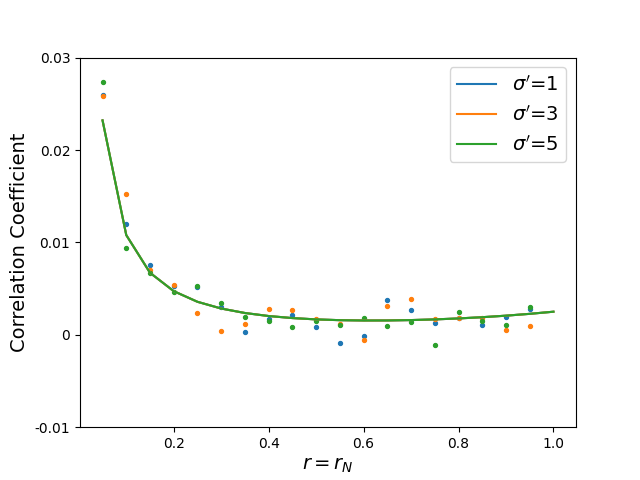}
         \label{q11s1:Nfr1}
         \caption{Scenario $(iii)$ with $N=200$ and $M=500000$}
         \label{S:q11simulationse}
     \end{subfigure}
     \hfill
     \begin{subfigure}[b]{0.355\textwidth}
         \centering         \includegraphics[width=\textwidth]{ 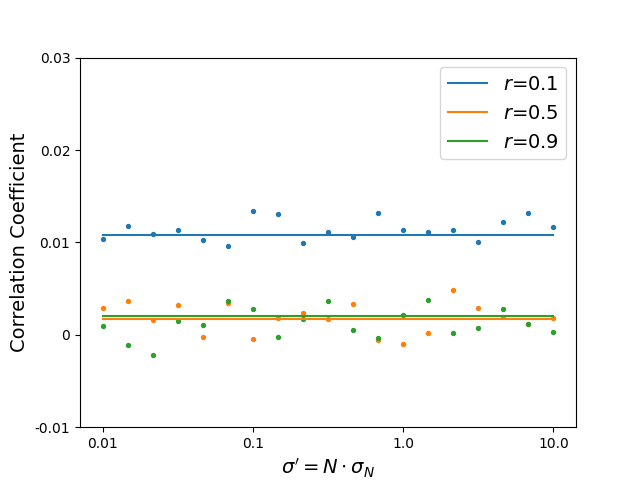}
         \label{q11s1:Nr1f}
         \caption{Scenario $(iii)$ with $N=200$ and $M=500000$}
         \label{S:q11simulationsf}
     \end{subfigure}
     \hspace*{\fill}

     \hspace*{\fill}
     \begin{subfigure}[b]{0.355\textwidth}
         \centering   \includegraphics[width=\textwidth]{ 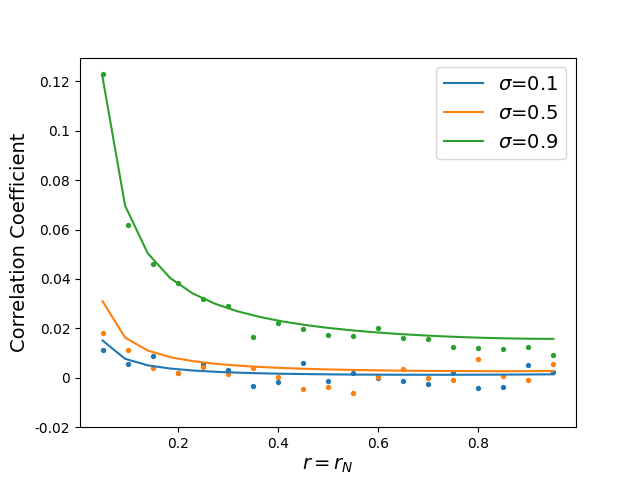}
         \caption{Scenario $(iv)$ with $500$ and $M=100000$}
         \label{S:q11simulationsg}
     \end{subfigure}
     \hfill
     \begin{subfigure}[b]{0.355\textwidth}
         \centering         \includegraphics[width=\textwidth]{ 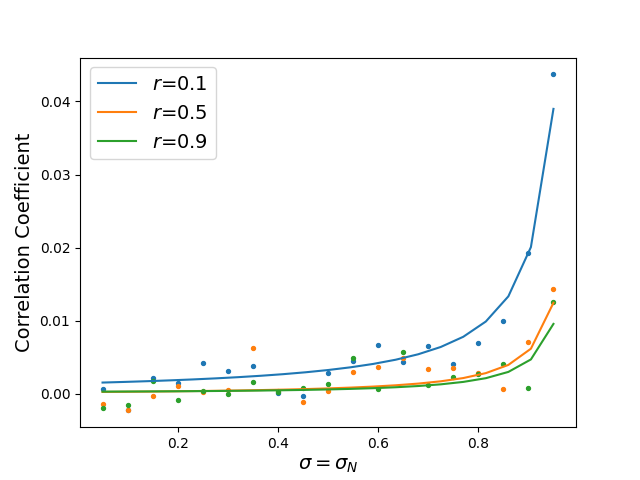}
         \caption{Scenario $(iv)$ with $N=1200$ and $M=150000$}
         \label{S:q11simulationsh}
     \end{subfigure}
     \hspace*{\fill}
\caption{Correlation coefficient for initial state $q_{11}$}
\label{S:q11simulations}
\end{figure}

\FloatBarrier

\medskip






\section{Discussion}\label{S:Discussion}
 





In this paper, we considered a diploid Wright-Fisher model with a single mating type and partial selfing, and studied the coalescence times at two loci under four scaling scenarios depending on the relative strengths of the selfing versus recombination. We established asymptotic formulas for the covariance  of the coalescence times under these four scaling scenarios, for all possible sampling configurations in the model.


We have not considered other effects such as selection and fluctuations in population size, and focuses on sample size two, which already yields interesting formulas. Our method can in principle be extended to study correlation for more than two samples, though the computation will be significantly more complicated. 

Our model and results themselves do not involve mutation, but they have direct implications to inference methods for genetic data. We demonstrate this by applying our results to obtain exact asymptotics for the variance of the Tajima's estimator for the population mutation rate under the infinite site model, following the approach in \cite{king2018non}. 
Our results may also have implications to linkage disequilibrium (LD), since the connection between LD and correlations of genealogies is understood \citep{mcvean2002genealogical}.

\subsection{Discrepancies for small values of \texorpdfstring{$N$}{TEXT}}

Our asymptotic formulas for the correlation coefficients require $N\to\infty$ and do not tell us how much they deviate from the corresponding correlation coefficients for finite $N$.

To assess such deviation, or the discrepancy between the two, we plot both of them as functions of $N$, against simulations, in each of Figures \ref{S:bias1} to \ref{S:bias8}. In all of Figures \ref{S:bias1} to \ref{S:bias8}, the number of independent trials is $M=50000$ for the discrete simulation, and  $s_N=0.95$ and $r_N=0.1$.



These plots demonstrate that a significant discrepancy can occur for small values of $N$; see for instance Figure \ref{S:bias2}. 
They also confirm that as $N$ gets larger, this discrepancy between the discrete simulations and the asymptotics formula gets smaller, providing additional confirmation of our analytical results.
They are also helpful for choosing the smallest value of $N$ to decrease runtime for the simulations in section \ref{S:q12sims}.
Increasing $N$ and the number of trials $M$ will decrease this discrepancy.

Note that  our analytical formula overestimate the corresponding simulated value for small $N$ in 
 Figures \ref{S:bias4} and \ref{S:bias4}, but 
 underestimate the corresponding simulated value in Figures
 \ref{S:bias1}, \ref{S:bias2}, \ref{S:bias5}, \ref{S:bias6}, \ref{S:bias7}, and \ref{S:bias8}. 
We do not have an explanation or intuition about why some of them are overestimations while other are underestimations. A possible explanation is that,  for smaller $N$, the lower order term (in $N$) of the asymptotic expressions of the correlation may not be ignored, and that these lower order term can be positive or negative. This can be a question for future investigation. 


\FloatBarrier

\begin{figure}\label{S:diffN}
    \captionsetup{justification=centering}
     \centering
     \hspace*{\fill}
     \begin{subfigure}[b] {0.35\textwidth}
         \centering       \includegraphics[width=\textwidth]{ 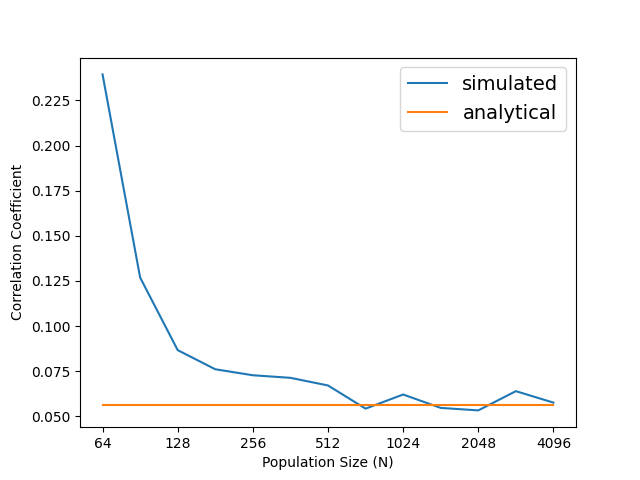}    
         \caption{Scenario $(i)$ state $q_{11}$ with $r_N=5/N$ and $s_N=50/N$} \label{S:bias1}
     \end{subfigure}
     \hfill
     \begin{subfigure}[b]{0.35\textwidth}
         \centering         \includegraphics[width=\textwidth]{ 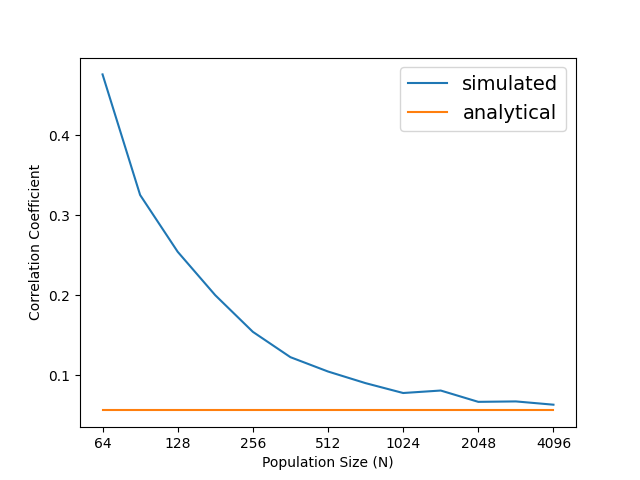}
         \caption{Scenario $(i)$ state $q_{12}$ with $r_N=5/N$ and $s_N=50/N$} \label{S:bias2}
     \end{subfigure}
     \centering
     \hspace*{\fill}

     \hspace*{\fill}
     \begin{subfigure}[b]{0.35\textwidth}
         \centering   \includegraphics[width=\textwidth]{ 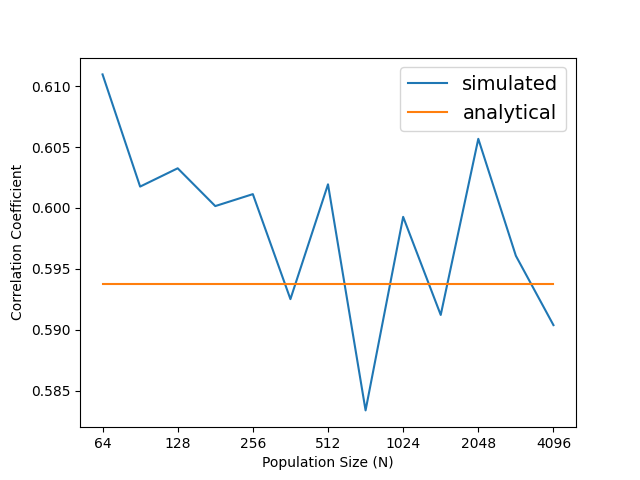}
         \caption{Scenario $(i)$ state $q_{11}$ with $r_N=5/N$ and $s_N=0.95$} \label{S:bias3}
     \end{subfigure}
     \hfill
     \begin{subfigure}[b]{0.35\textwidth}
         \centering       \includegraphics[width=\textwidth]{ 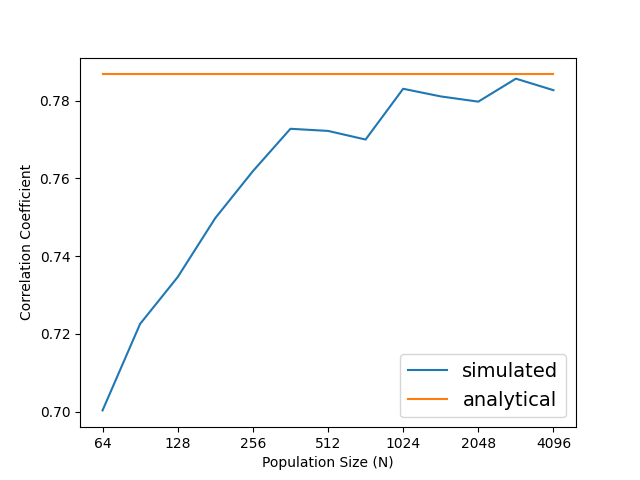}
         \caption{Scenario $(i)$ state $q_{12}$ with $r_N=5/N$ and $s_N=0.95$} \label{S:bias4}
     \end{subfigure}
     \hspace*{\fill}
     
     \hspace*{\fill}
     \begin{subfigure}[b]{0.35\textwidth}
         \centering   \includegraphics[width=\textwidth]{ 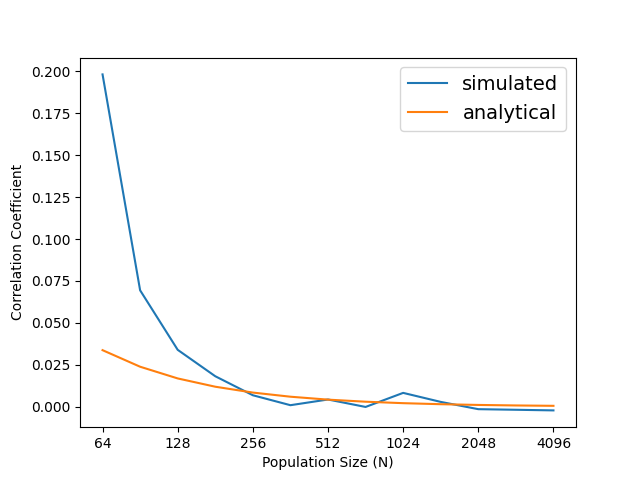}
         \caption{Scenario $(iii)$ state $q_{11}$ with $r_N=0.1$ and $s_N=50/N$} \label{S:bias5}
     \end{subfigure}
     \hfill
     \begin{subfigure}[b]{0.35\textwidth}
         \centering         \includegraphics[width=\textwidth]{ 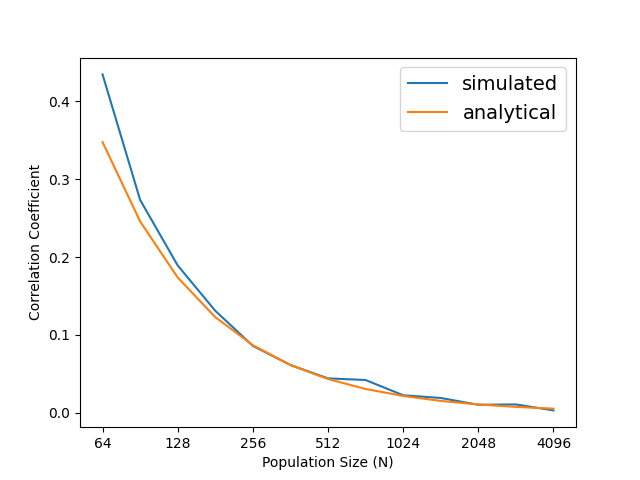}
         \caption{Scenario $(iii)$ state $q_{12}$ with $r_N=0.1$ and $s_N=50/N$} \label{S:bias6}
     \end{subfigure}
     \hspace*{\fill}

     \hspace*{\fill}
     \begin{subfigure}[b]{0.35\textwidth}
         \centering   
         \includegraphics[width=\textwidth]{ 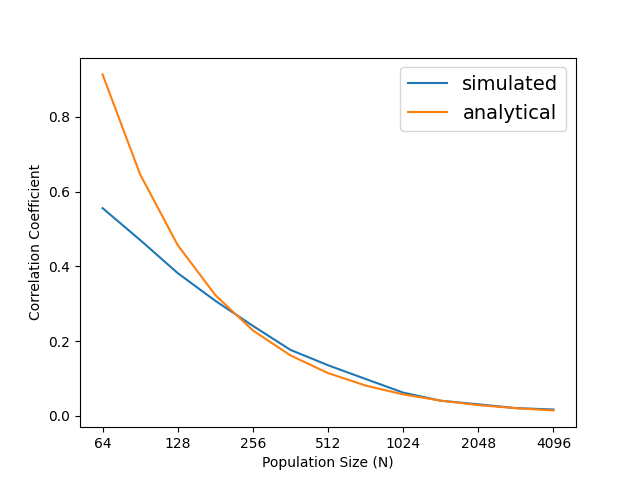}
         \caption{Scenario $(iv)$ state $q_{11}$ with $r_N=0.1$ and $s_N=0.95$} \label{S:bias7}
     \end{subfigure}
     \hfill
     \begin{subfigure}[b]{0.35\textwidth}
         \centering         \includegraphics[width=\textwidth]{ 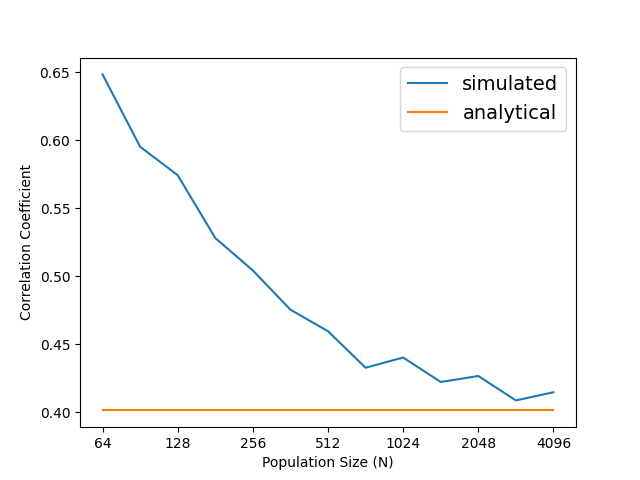}
         \caption{Scenario $(iv)$ state $q_{12}$ with $r_N=0.1$ and $s_N=0.95$} \label{S:bias8}
     \end{subfigure}
     \hspace*{\fill}

     \caption{Correlation coefficient as a function of population size $N$}
\end{figure}

\FloatBarrier

\subsection{Computational complexity}
Each plotted point in  simulations for Figures \ref{S:q12simulationsa}-\ref{S:q12simulationsh}, \ref{S:q11simulationsa}-\ref{S:q11simulationsh}, and \ref{S:bias1}-\ref{S:bias8} has complexity  $O(N M)$ in expectation, where $M$ is the number of repeated independent trials in the Monte Carlo method (i.e. the number of times we run Algorithm \ref{A:algorithm}), and $N$ is the population size. We said ``in expectation" because the  complexity of Algorithm \ref{A:algorithm} is  random.
Algorithm \ref{A:algorithm} randomly simulates the ancestral process of two loci until both pairs coalesce by selecting states corresponding to the previous generation according to the explicit transition probabilities in matrix $\bf \Pi_N^*$ defined in Section \ref{S:simulations}. Hence, the complexity of this algorithm is of the order $O(\max\{T_i, T_j\})$, which is random. Note that  $T_i \leq\max\{T_i, T_j\}\leq T_i+T_j$.
Hence,  by Lemma \ref{L:ETi}, the complexity of a single execution of Algorithm \ref{A:algorithm} is $O(N)$ in expectation.

Each plotted point in simulations \ref{S:q12simulationsa}-\ref{S:q12simulationsh}, \ref{S:q11simulationsa}-\ref{S:q11simulationsh}, and \ref{S:bias1}-\ref{S:bias8} is the average of $M$ executions of Algorithm \ref{A:algorithm} for fixed initial parameters $N$, $s_N$, $r_N$. Each plot has $\approx 50$ plotted points for a total expected complexity of $O(50 M N)$.

The standard error of the Pearson correlation prediction for a single point plotted in Figures \ref{S:q12simulationsa}-\ref{S:q12simulationsh} and Figures \ref{S:q11simulationsa}-\ref{S:q11simulationsh} is $O(M^{-1/2})$. Thus, to achieve a standard error $O(N^{-1})$, one must increase the number of trials $M$ by $O(N^2)$. Figures \ref{S:q12simulationsa}-\ref{S:q12simulationsh} and \ref{S:q11simulationsa}-\ref{S:q11simulationsh} display more variation when the correlation coefficient is of order $O(N^{-1})$ than when the correlation coefficient is of order $O(1)$, because we need approximately $O(N^2)$ more trials to produce standard error of  $O(N^{-1})$. The cases when this happens are  scenario 
 $(iii)$ for states $q_{11}$ and $q_{12}$, and for scenario $(iv)$ for state $q_{11}$.

\section*{Acknowledgements}
This work was supported by the National Science Foundation grant DMS-2051032, when David Kogan was an REU student at Indiana University during summer 2022. It is also supported by the NSF grant DMS-2152103, and the Office of Naval Research grant N00014-20-1-2411 to Wai-Tong (Louis) Fan.



\appendix
\section{States  \texorpdfstring{$\mathcal S=\{q_i\}_{i=1}^{12}$}{TEXT}}\label{A:States}
Here we list all possible states for the two pairs of gene copies. Following \cite{king2018non} we obtain $12$ non-coalescent states $\{q_i\}_{i=1}^{12}$. We do not keep track of the exact labels of the gene copies, but only whether they are in the same individuals and whether coalescence occurred. Adopting the notation in  \cite[Appendix]{king2018non} and supplementing with a pictorial representation, we list these 12 states below:  
\begin{align*}
q_1 &= \{1, 1, 2, 2\} &&=
    \begin{pmatrix}
        {\color{red} \bullet } &  {\color{white} \bullet }
        \\[-0.2cm]
         {\color{white} \blacktriangle }  & {\color{white} \blacktriangle } 
    \end{pmatrix}
    \begin{pmatrix}
     {\color{red} \bullet } & {\color{white} \bullet }
        \\[-0.2cm]
         {\color{white} \blacktriangle } &  {\color{white} \blacktriangle }
    \end{pmatrix}
    \begin{pmatrix}
        &   
        \\[-0.2cm]
        {\color{OliveGreen} \blacktriangle } &  {\color{white} \blacktriangle }
    \end{pmatrix}
    \begin{pmatrix}
        &   
        \\[-0.2cm]
        {\color{OliveGreen} \blacktriangle } &  {\color{white} \blacktriangle }
    \end{pmatrix}
\\
q_2 &=\{(1, 1), 2, 2\}&&=
    \begin{pmatrix}
        {\color{red} \bullet }&  {\color{red} \bullet }
        \\[-0.2cm]
        {\color{white} \blacktriangle }   &  {\color{white} \blacktriangle }
    \end{pmatrix}
    \begin{pmatrix}
        &  
        \\[-0.2cm]
         {\color{OliveGreen} \blacktriangle } &  {\color{white} \blacktriangle }
    \end{pmatrix}
    \begin{pmatrix}
        &   
        \\[-0.2cm]
         {\color{OliveGreen} \blacktriangle } & {\color{white} \blacktriangle }
    \end{pmatrix}
\\
q_3 &= \{1, 1, (2, 2)\}&&=
    \begin{pmatrix}
        {\color{red} \bullet } & 
        \\[-0.2cm]
        {\color{white} \blacktriangle } &  {\color{white} \blacktriangle }
    \end{pmatrix}
    \begin{pmatrix}
          {\color{red} \bullet } &
        \\[-0.2cm]
        {\color{white} \blacktriangle }  & {\color{white} \blacktriangle }
    \end{pmatrix}
    \begin{pmatrix}
         &   
        \\[-0.2cm]
         {\color{OliveGreen} \blacktriangle } & {\color{OliveGreen} \blacktriangle }
    \end{pmatrix}
\\
q_4&=\{(1, 2), 1, 2\}&& =
    \begin{pmatrix}
        {\color{red} \bullet }&  
        \\[-0.2cm]
        {\color{white} \blacktriangle }  & {\color{OliveGreen} \blacktriangle }
    \end{pmatrix}
    \begin{pmatrix}
        {\color{red} \bullet } & {\color{white} \bullet }  
        \\[-0.2cm]
        {\color{white} \blacktriangle }  & {\color{white} \blacktriangle }
    \end{pmatrix}
    \begin{pmatrix}
        &    
        \\[-0.2cm]
         {\color{OliveGreen} \blacktriangle } & {\color{white} \blacktriangle }
    \end{pmatrix}
\\
q_5 &= \{(1, 1), (2, 2)\}&&=
    \begin{pmatrix}
        {\color{red} \bullet }&  {\color{red} \bullet }
        \\[-0.2cm]
        {\color{white} \blacktriangle }  &  {\color{white} \blacktriangle }
    \end{pmatrix}
    \begin{pmatrix}
         &   
        \\[-0.2cm]
         {\color{OliveGreen} \blacktriangle } & {\color{OliveGreen} \blacktriangle }
    \end{pmatrix}
\\
q_6 &=  \{(1, 2), (1, 2)\}&&=
   \begin{pmatrix}
        {\color{red} \bullet }&   
        \\[-0.2cm]
        {\color{white} \blacktriangle }  & {\color{OliveGreen} \blacktriangle }
    \end{pmatrix}
    \begin{pmatrix}
        {\color{red} \bullet }&   
        \\[-0.2cm]
        {\color{white} \blacktriangle }  & {\color{OliveGreen} \blacktriangle }
    \end{pmatrix}
\\
q_7 &=  \{12, 1, 2\}&&=
\begin{pmatrix}
        {\color{red} \bullet }&   
        \\[-0.2cm]
         {\color{OliveGreen} \blacktriangle } &  {\color{white} \blacktriangle }
    \end{pmatrix}
    \begin{pmatrix}
        {\color{red} \bullet }&  {\color{white} \bullet }  
        \\[-0.2cm]
        {\color{white} \blacktriangle }  &  {\color{white} \blacktriangle }
    \end{pmatrix}
    \begin{pmatrix}
       &   
        \\[-0.2cm]
          {\color{OliveGreen} \blacktriangle } & {\color{white} \blacktriangle }
    \end{pmatrix}
\\
q_8 &= \{(12, 1), 2\}&&=
   \begin{pmatrix}
        {\color{red} \bullet}&  {\color{red} \bullet}
        \\[-0.2cm]
         {\color{OliveGreen} \blacktriangle} &  {\color{white} \blacktriangle }
    \end{pmatrix}
    \begin{pmatrix}
        &   
        \\[-0.2cm]
         {\color{OliveGreen} \blacktriangle } & {\color{white} \blacktriangle }
    \end{pmatrix}
\\
q_9 &=  \{(12, 2), 1\}&&=
   \begin{pmatrix}
        {\color{red} \bullet }&   
        \\[-0.2cm]
         {\color{OliveGreen} \blacktriangle } & {\color{OliveGreen} \blacktriangle }
    \end{pmatrix}
    \begin{pmatrix}
        {\color{red} \bullet }& {\color{white} \bullet }   
        \\[-0.2cm]
        {\color{white} \blacktriangle }  &  {\color{white} \blacktriangle }
    \end{pmatrix}
\\
q_{10} &= \{12, (1, 2)\}&&=
    \begin{pmatrix}
        {\color{red} \bullet }&   
        \\[-0.2cm]
         {\color{OliveGreen} \blacktriangle } & {\color{white} \blacktriangle } 
    \end{pmatrix}
    \begin{pmatrix}
        {\color{red} \bullet }&  
        \\[-0.2cm]
        {\color{white} \blacktriangle } & {\color{OliveGreen} \blacktriangle }
    \end{pmatrix}
\\
q_{11} &=  \{12, 12\}&&=
    \begin{pmatrix}
        {\color{red} \bullet }&   
        \\[-0.2cm]
         {\color{OliveGreen} \blacktriangle } & {\color{white} \blacktriangle }  
    \end{pmatrix}
    \begin{pmatrix}
          {\color{red} \bullet } &
        \\[-0.2cm]
           {\color{OliveGreen} \blacktriangle } & {\color{white} \blacktriangle }
    \end{pmatrix}
\\
q_{12} &= \{(12, 12)\}&&=
 \begin{pmatrix}
        {\color{red} \bullet }&  {\color{red} \bullet}
        \\[-0.2cm]
         {\color{OliveGreen} \blacktriangle } & {\color{OliveGreen} \blacktriangle }
    \end{pmatrix}.
\end{align*}
For each of the above states, ${\color{red} \bullet }$ and ${\color{OliveGreen} \blacktriangle }$ represent sampled gene copies at locus $i$ and $j$ respectively.  Big parenthesis $``\left(\right)"$ are used to denote an individual. Each column represents a chromosome with two loci, so that each row represents a locus.
 The  chromosomes are not labeled, hence there is no distinguishing between left and right for the two chromosomes within an individual. For example, the following equivalences hold
$$ \begin{pmatrix}
        {\color{red} \bullet }& {\color{white} \bullet }  
        \\[-0.2cm]
           &  
    \end{pmatrix}
    =
    \begin{pmatrix}
       {\color{white} \bullet } & {\color{red} \bullet }   
        \\[-0.2cm]
       {\color{white} \blacktriangle }    &  {\color{white} \blacktriangle }
    \end{pmatrix},\quad 
 \begin{pmatrix}
          {\color{red} \bullet } & {\color{white} \bullet }
        \\[-0.2cm]
           {\color{OliveGreen} \blacktriangle } & {\color{white} \blacktriangle }
    \end{pmatrix}
    =
 \begin{pmatrix}
         {\color{white} \bullet } &  {\color{red} \bullet } 
        \\[-0.2cm]
        {\color{white} \blacktriangle } &  {\color{OliveGreen} \blacktriangle } 
    \end{pmatrix} ,\quad   
    \begin{pmatrix}
        {\color{red} \bullet }&  
        \\[-0.2cm]
        {\color{white} \blacktriangle } & {\color{OliveGreen} \blacktriangle }
    \end{pmatrix}
    =
    \begin{pmatrix}
        &{\color{red} \bullet } 
        \\[-0.2cm]
          {\color{OliveGreen} \blacktriangle }&{\color{white} \blacktriangle }
    \end{pmatrix},
$$
and so we used only one of each in the above list of the 12 states.

We also denote by $q_0$ the state in which coalescence occurred in one or both loci. This state $q_0$ is an absorbing state of the Markov chain for the coalescence process of the two pairs of gene copies.

\section{The \texorpdfstring{$13\times13$}{TEXT} transition matrix \texorpdfstring{${\bf \Pi}_N$}{TEXT}}\label{A:Matrix}


The dynamics of the two pairs of lineages at the two loci can be described by a discrete-time Markov chain with  13 states $\mathcal S \cup\{q_0\}$ and the 
$13\times 13$ one-step transition matrix ${\bf \Pi}_N$ below. 

We now explain how the matrix is computed from the Wright-Fisher model through a few specific examples. 
\begin{example} \label{ex:p2_2}\rm  The transition probability from $q_2=\{(1, 1), 2, 2\}$ to $q_2$ is obtained as follows. We note that the two loci $i$ (${\color{red} \bullet}$ in \ref{A:States}) must stay in the same individual but not coalesce yet, which happens with probability $s_N/2$. The two loci $j$ (${\color{OliveGreen} \blacktriangle}$ in \ref{A:States}) must belong to two different individuals that are different from the individual chosen by loci $i$, which has probability $\frac{N-1}{N}\frac{N-2}{N}$.
Thus, we have that the transition probability from $q_2$ to $q_2$ is $\frac{s_N(N-1)(N-2)}{2N^2}$.
\end{example}

\begin{example} \label{ex:p2_3}\rm  The transition probability from $q_2=\{(1, 1), 2, 2\}$ to $q_3
=\{1, 1, (2, 2)\}$ is obtained as follows. We note that the two loci $i$ (${\color{red} \bullet}$ in \ref{A:States}) must end up in different individuals, which happens with probability $1-s_N$. The two loci $j$ (${\color{OliveGreen} \blacktriangle}$ in \ref{A:States}) must end up in same individual that is different from the two individuals with loci $i$, but not coalesce. The probability the two loci $j$ end up in the same individual and do not coalesce is $\frac{1}{2N}$. Additionally, the probability that this individual is different from the two individuals containing the loci $i$ is $1-\frac{2}{N}$. This is because the two loci $i$ cannot end up in the same individual as they did not come from selfing, and we have that $(1-\frac{1}{N})\cdot\frac{N-2}{N-1}=1-\frac{2}{N}$. Thus the transition probability from state $q_2$ to $q_3$ is $(1-s_N)(1-\frac{2}{N})\cdot \frac{1}{2N}= \frac{(1-s_N)(N-2)}{2N^2}$.
\end{example}

\begin{example}\label{ex:p6_4}\rm  The transition probability from $q_6=\{(1, 2), (1, 2)\}$ to $q_4
=\{(1, 2), 1, 2)\}$ is obtained as follows. Denote the two individuals in state $q_6$ by $A$ and $B$. There are two different cases. The first case is when one of the two individuals in state $q_6$ undergoes selfing, and the other doesn't undergo selfing. The second case is when both $A$ and $B$ do not undergo selfing. We can ignore the case when both undergo selfing because this will result in loci in at most two individuals, but state $q_4$ has loci in three individuals. 

For the first case the probability that $A$ undergoes selfing, and $B$ does not is $s_N(1-s_N)$. The probability that the two loci in individual $A$ do not end up on the same chromosome is $1/2$. The two loci from $B$ must not end up in the same individual as the loci from $A$. As in Example \ref{ex:p2_3} the probability of this is $(1-\frac{1}{N})\cdot\frac{N-2}{N-1}=\frac{N-2}{N}$. We have the same probability of transitioning to state $q_4$ from $q_6$ if $B$ undergoes selfing, and $A$ doesn't. Thus the probability of the first case is $2\cdot \frac{1}{2}\cdot s_N(1-s_N)\cdot\frac{N-2}{N} = \frac{s_N(1-s_N)(N-2)}{N}$

Now consider the second case. Both $A$ and $B$ do not undergo selfing with probability $(1-s_N)^2$. Now we need to select one locus $i$ and one locus $j$ to end up in the same individual. Note that only two such pairs exists because we need to pick locus $i$ and $j$ from different individuals (since no selfing). The probability that given a locus $i$ and locus a $j$ pair they end up in the same individual on different chromosomes as in $q_4$ is $\frac{1}{2N}$. Note the other two loci cannot end up in the same individual as the one with both a locus $i$ and locus $j$ since both individuals in $q_6$ did not undergo selfing. Still, the other two loci must not end up in the same individual. This happens with probability $\frac{N-2}{N-1}$. Thus the total probability of the second case is $2\cdot \frac{1}{2N}\cdot (1-s_N)^2\cdot \frac{N-2}{N-1}$. Thus the total transition probability from $q_6$ to $q_4$ is $\frac{s_N(1-s_N)(N-2)}{N}+\frac{(1-s_N)^2(N-2)}{N(N-1)
}$.

\end{example}

\FloatBarrier
\begin{table}[ht]\rm \label{M:transition_matrix}
\renewcommand{\arraystretch}{1.5}
\centering
\scalebox{0.75}{
\begin{tabular}{|c| c c c c c c|}
\hline
     & $q_1$ & $q_2$ & $q_3$ & $q_4$ & $q_5$ & $q_6$ \\
     \hline
          \\
     $q_1$ & $\frac{(N-1)(N-2)(N-3)}{N^3}$ & $\frac{(N-1)(N-2)}{2N^3}$ & $\frac{(N-1)(N-2)}{2N^3}$ & $\frac{2(N-1)(N-2)}{N^3}$ & $\frac{N-1}{4N^3}$ & $\frac{N-1}{2N^3}$ 
          \\
          \\
     $q_2$ & $\frac{(1 - s_N)(N-2)(N-3)}{N^2}$ & $\frac{s_N(N-1)(N-2)}{2N^2}$ & $\frac{(1 - s_N)(N-2)}{2N^2}$ & $\frac{2(1-s_N)(N-2)}{N^2}$ & $\frac{s_N(N-1)}{4N^2}$ & $\frac{(1-s_N)}{2N^2}$
          \\
          \\
     $q_3$ & $\frac{(1 - s_N)(N-2)(N-3)}{N^2}$ & $\frac{(1 - s_N)(N-2)}{2N^2}$ & $\frac{s_N(N-1)(N-2)}{2N^2}$ & $\frac{2(1-s_N)(N-2)}{N^2}$ & $\frac{s_N(N-1)}{4N^2}$ & $\frac{(1-s_N)}{2N^2}$
     \\
          \\
     $q_4$ & $\frac{(1-s_N)(N-2)(N-3)}{N^2}$ & $\frac{(1-s_N)(N-2)}{2N^2}$ & $\frac{(1-s_N)(N-2)}{2N^2}$ & $\frac{s_N(N-1)(N-2)}{2N^2} + \frac{3(1-s_N)(N-2)}{2N^2}$ & $\frac{1-s_N}{4N^2}$ & $\frac{1-s_N}{4N^2}+\frac{s_N(N-1)}{4N^2}$
     \\
          \\
     $q_5$ & $\frac{(1-s_N)^2(N-2)(N-3)}{N(N-1)}$ & $\frac{s_N(1-s_N)(N-2)}{2N}$ & $\frac{s_N(1-s_N)(N-2)}{2N}$ & $\frac{2(1-s_N)^2(N-2)}{N(N-1)}$ & $\frac{s_N^2(N-1)}{4N}$ & $\frac{(1-s_N)^2}{2N(N-1)}$
          \\
     \\
     $q_6$ & $\frac{(1-s_N)^2(N-2)(N-3)}{N(N-1)}$ & $\frac{(1-s_N)^2(N-2)}{2N(N-1)}$ & $\frac{(1-s_N)^2(N-2)}{2N(N-1)}$ & $\frac{s_N(1-s_N)(N-2)}{N} + \frac{(1-s_N)^2(N-2)}{N(N-1)}$ & $\frac{(1-s_N)^2}{4N(N-1)}$ & $\frac{(1-s_N)^2}{4N(N-1)}+\frac{s_N^2(N-1)}{4N}$
     \\
          \\
     $q_7$ & 0 & 0 & 0 & $\frac{r_N(N-1)(N-2)}{N^2}$ & 0 & $\frac{r_N}{2N}(1-\frac{1}{N})$
     \\
          \\
     $q_8$ & 0 & 0 & 0 & $\frac{r_N(1-s_N)(N-2)}{N}$ & 0 & $\frac{r_N(1-s_N)}{2N}$
     \\
          \\
     $q_9$ & 0 & 0 & 0 & $\frac{r_N(1-s_N)(N-2)}{N}$ & 0 & $\frac{r_N(1-s_N)}{2N}$
     \\
          \\
     $q_{10}$ & 0 & 0 & 0 & $\frac{r_N(1-s_N)(N-2)}{N}$ & 0 & $\frac{r_Ns_N(N-1)}{2N}$
     \\
          \\
     $q_{11}$ & 0 & 0 & 0 & 0 & 0 & $\frac{r_N^2(N-1)}{N}$
     \\
          \\
     $q_{12}$ & 0 & 0 & 0 & 0 & 0 & $r_N^2(1-s_N)$
     \\
          \\
     $q_{0}$ & 0 & 0 & 0 & 0 & 0 & 0
     \\
     \hline
\end{tabular}
}
\end{table}


\begin{table}[ht]
\renewcommand{\arraystretch}{1.5}
\begin{center}
\scalebox{0.65}{
\begin{tabular} {|c| c c c c c c c|}
\hline
     & $q_7$ & $q_8$ & $q_9$ & $q_{10}$ & $q_{11}$ & $q_{12}$ & $s_0$ 
     \\
     \hline
          \\
     $q_1$ & $\frac{2(N-1)(N-2)}{N^3}$ & $\frac{N-1}{N^3}$ & $\frac{N-1}{N^3}$ & $\frac{N-1}{N^3}$ & $\frac{N-1}{2N^3}$ & $\frac{1}{4N^3}$ & $\frac{1}{N}-\frac{1}{4N^2}$
          \\
          \\
     $q_2$ & $\frac{2(1-s_N)(N-2)}{N^2}$ & $\frac{s_N(N-1)}{N^2}$ & $\frac{(1-s_N)}{N^2}$ & $\frac{(1-s_N)}{N^2}$ & $\frac{(1-s_N)}{2N^2}$ & $\frac{s_N}{4N^2}$ & $\frac{s_N}{2}+\frac{2-s_N}{4N}$
          \\
     \\
     $q_3$ & $\frac{2(1-s_N)(N-2)}{N^2}$ & $\frac{(1-s_N)}{N^2}$ & $\frac{s_N(N-1)}{N^2}$ & $\frac{(1-s_N)}{2N^2}$ & $\frac{(1-s_N)}{N^2}$ & $\frac{s_N}{4N^2}$ & $\frac{s_N}{2}+\frac{2-s_N}{4N}$
          \\
     \\
     $q_4$ & $\frac{s_N(N-1)(N-2)}{2N^2} + \frac{3(1-s_N)(N-2)}{2N^2}$ & $\frac{s_N(N-1)}{2N^2}+\frac{(1-s_N)}{2N^2}$ & $\frac{s_N(N-1)}{2N^2}+\frac{(1-s_N)}{2N^2}$ & $\frac{s_N(N-1)}{2N^2}+\frac{(1-s_N)}{2N^2}$ & $\frac{s_N(N-1)}{4N^2}+\frac{(1-s_N)}{4N^2}$ & $\frac{s_N}{4N^2}$ & $\frac{1}{N}-\frac{1}{4N^2}$
     \\
          \\
     $q_5$ & $\frac{2(1-s_N)^2(N-2)}{N(N-1)}$ & $\frac{s_N(1-s_N)}{N}$ & $\frac{s_N(1-s_N)}{N}$ & $\frac{(1-s_N)^2}{N(N-1)}$ & $\frac{(1-s_N)^2}{2N(N-1)}$ & $\frac{s_N^2}{4N}$ & $s_N-\frac{s_N^2}{4}$
     \\
          \\
     $q_6$ & $\frac{s_N(1-s_N)(N-2)}{N}+ \frac{(1-s_N)^2(N-2)}{N(N-1)}$ & $\frac{s_N(1-s_N)}{N}$ & $\frac{s_N(1-s_N)}{N}$ & $\frac{s_N^2(N-1)}{2N}+\frac{(1-s_N)^2}{2N(N-1)}$ & $\frac{s_N^2(N-1)}{4N}+\frac{(1-s_N)^2}{4N(N-1)}$ & $\frac{s_N^2}{4N}$ & $\frac{1}{N}-\frac{s_N^2}{4N}-\frac{(1-s_N)^2}{4N(N-1)}$
     \\
          \\
     $q_7$ & $\frac{(1-r_N)(N-1)(N-2)}{N^2}$ & $\frac{N-1}{2N^2}$ & $\frac{N-1}{2N^2}$ & $\frac{N-1}{2N^2}$ & $\frac{(1-r)(N-1)}{2N^2}$ & $\frac{1}{4N^2}$ & $\frac{1}{N}-\frac{1}{4N^2}$
     \\
          \\
     $q_8$ & $\frac{(1-r_N)(1-s_N)(N-2)}{N}$ & $\frac{s_N(N-1)}{2N}$ & $\frac{1-s_N}{2N}$ & $\frac{1-s_N}{2N}$ & $\frac{(1-r_N)(1-s_N)}{2N}$ & $\frac{s_N}{4N}$ & $\frac{s_N}{2}+\frac{2-s_N}{4N}$
     \\
          \\
     $q_9$ & $\frac{(1-r_N)(1-s_N)(N-2)}{N}$ & $\frac{1-s_N}{2N}$ & $\frac{s_N(N-1)}{2N}$ & $\frac{1-s_N}{2N}$ & $\frac{(1-r_N)(1-s_N)}{2N}$ & $\frac{s_N}{4N}$ & $\frac{s_N}{2}+\frac{2-s_N}{4N}$
     \\
          \\
     $q_{10}$ & $\frac{(1-r_N)(1-s_N)(N-2)}{N}$ & $\frac{1-s_N}{2N}$ & $\frac{1-s_N}{2N}$ & $\frac{s_N(N-1)}{2N}$ & $\frac{(1-r_N)s_N(N-1)}{2N}$ & $\frac{s_N}{4N}$ & $\frac{4-s_N}{4N}$
     \\
          \\
     $q_{11}$ & 0 & 0 & 0 & $\frac{2r_N(1-r_N)(N-1)}{N}$ & $\frac{(1-r_N)^2(N-1)}{N}$ & $\frac{(1-r_N)^2+r_N^2}{2N}$ & $\frac{-2r_N^2+2r_N+1}{2N}$
     \\
          \\
     $q_{12}$ & 0 & 0 & 0 & $2r_N(1-r_N)(1-s_N)$ & $(1-r_N)^2(1-s_N)$ & $\frac{((1-r_N)^2+r_N^2)s_N}{2}$ & $\frac{s_N(-2r_N^2+2r_N+1)}{2}$
     \\
          \\
     $q_0$ & 0 & 0 & 0 & 0 & 0 & 0 & 1
     \\
\hline
\end{tabular}
}
\end{center}
\end{table}
\FloatBarrier

We note that the rows for the states $q_2, q_3$ and $q_8, q_9$ are almost the same. This is because for all other states the loci are interchangable meaning that if we swapped locus $i$ with locus $j$ the state would be the same. On the other hand swapping locus $i$ with locus $j$ in state $q_2$ would result in state $q_3$, and vice-versa (similar for $q_8$ and $q_9$). This means that states $q_2$ and $q_3$ should have the same transition probabilities to all other states except $q_2, q_3, q_8$ and $q_9$. Additionally, the transition probability from $q_2$ to $q_2$ and $q_3$ should equal the transition probability from $q_3$ to $q_3$ and $q_2$ respectively (same argument applies for $q_8$ and $q_9$). This is why the rows for the states $q_2, q_3$ and $q_8, q_9$ are almost the same.

\subsection{Transition matrix  \texorpdfstring{${\bf \Pi}_N$}{TEXT} when \texorpdfstring{$s_N=1$}{TEXT} (total selfing)}

In the case when $s_N=1$, 
this matrix becomes the following. We expect that some of the transition probabilities become $0$. Specifically, once a group of gene copies are found in the same individual, they must remain in the same individual. For example, the state $q_{12}$ can only transition to $q_{12}$ or the coalescent state.

\medskip

\begin{center}
\scalebox{0.85}{
\begin{tabular}{|c|cccccc|}
\hline
& $q_1$ & $q_2$ & $q_3$ & $q_4$ & $q_5$ & $q_6$
\\
\hline
$q_1$ & $\frac{(N-1)(N-2)(N-3)}{N^3}$ & $\frac{(N-1)(N-2)}{2N^3}$ & $\frac{(N-1)(N-2)}{2N^3}$ & $\frac{2(N-1)(N-2)}{N^3}$ & $\frac{N-1}{4N^3}$ & $\frac{N-1}{2N^3}$
\\
$q_2$ & $0$ & $\frac{(N-1)(N-2)}{2N^2}$ & $0$ & $0$ & $\frac{N-1}{4N^2}$ & $0$ 
\\
$q_3$ & $0$ & $0$ & $\frac{(N-1)(N-2)}{2N^2}$ & $0$ & $\frac{N-1}{4N^2}$ & $0$
\\
$q_4$ & $0$ & $0$ & $0$ & $\frac{(N-1)(N-2)}{2N^2}$ & $0$ & $\frac{N-1}{4N^2}$
\\
$q_5$ & $0$ & $0$ & $0$ & $0$ & $\frac{N-1}{4N}$ & $0$
\\
$q_6$ & $0$ & $0$ & $0$ & $0$ & $0$ & $\frac{N-1}{4N}$
\\
$q_7$ & $0$ & $0$ & $0$ & $\frac{r_N(N-1)(N-2)}{N^2}$ & $0$ & $\frac{r_N(N-1)}{2N^2}$
\\
$q_8$ & $0$ & $0$ & $0$ & $0$ & $0$ & $0$
\\
$q_9$ & $0$ & $0$ & $0$ & $0$ & $0$ & $0$
\\
$q_{10}$ & $0$ & $0$ & $0$ & $0$ & $0$ & $\frac{r_N(N-1)}{2N}$
\\
$q_{11}$ & $0$ & $0$ & $0$ & $0$ & $0$ & $\frac{r_N^2(N-1)}{N^2}$
\\
$q_{12}$ & $0$ & $0$ & $0$ & $0$ & $0$ & $0$
\\
$q_0$ & $0$ & $0$ & $0$ & $0$ & $0$ & $0$
\\
\hline
\end{tabular}
}
\end{center}

\medskip

\begin{center}
\scalebox{0.85}{
\begin{tabular}{|c|ccccccc|}
\hline
& $q_7$ & $q_8$ & $q_9$ & $q_{10}$ & $q_{11}$ & $q_{12}$ & $q_0$
\\
\hline
$q_1$ & $\frac{2(N-1)(N-2)}{N^3}$ & $\frac{N-1}{N^3}$ & $\frac{N-1}{N^3}$ & $\frac{N-1}{N^3}$ & $\frac{N-1}{2N^3}$ & $\frac{1}{4N^3}$ & $\frac{4N-1}{4N^2}$
\\
$q_2$ & $0$ & $\frac{N-1}{N^2}$ & $0$ & $0$ & $0$ & $\frac{1}{4N^2}$ & $\frac{2N+1}{4N}$
\\
$q_3$ & $0$ & $0$ & $\frac{N-1}{N^2}$ & $0$ & $0$ & $\frac{1}{4N^2}$ & $\frac{2N+1}{4N}$
\\
$q_4$ & $\frac{(N-2)(N-1)}{2N^2}$ & $\frac{N-1}{2N^2}$ & $\frac{N-1}{2N^2}$ & $\frac{N-1}{2N^2}$ & $\frac{N-1}{4N^2}$ & $\frac{1}{4N^2}$ & $\frac{4N-1}{4N^2}$
\\
$q_5$ & $0$ & $0$ & $0$ & $0$ & $0$ & $\frac{1}{4N}$ & $\frac{3}{4}$
\\
$q_6$ & $0$ & $0$ & $0$ & $\frac{N-1}{2N}$ & $\frac{N-1}{4N}$ & $\frac{1}{4N}$ & $\frac{3}{4N}$
\\
$q_7$ & $\frac{(N-1)(N-2)(1-r_N)}{N^2}$ & $\frac{N-1}{2N^2}$ & $\frac{N-1}{2N^2}$ & $\frac{N-1}{2N^2}$ & $\frac{(N-1)(1-r_N)}{2N^2}$ & $\frac{1}{4N^2 }$ & $\frac{4N-1}{4N^2}$
\\
$q_8$ & $0$ & $\frac{N-1}{2N}$ & $0$ & $0$ & $0$ & $\frac{1}{4N}$ & $\frac{1+2N}{4N}$
\\
$q_9$ & $0$ & $0$ & $\frac{N-1}{2N}$ & $0$ & $0$ & $\frac{1}{4N}$ & $\frac{1+2N}{4N}$
\\
$q_{10}$ & $0$ & $0$ & $0$ & $\frac{N-1}{2N}$ & $\frac{(N-1)(1-r_N)}{2N}$ & $\frac{1}{4N}$ & $\frac{3}{4N}$
\\
$q_{11}$ & $0$ & $0$ & $0$ & $\frac{2r_N(1-r_N)(N-1)}{N}$ & $\frac{(N-1)(1-r_N)^2}{N}$  & $\frac{(1-r_N)^2+r_N^2}{2N}$ & $\frac{1+2r_N-2r_N^2}{2N}$
\\
$q_{12}$ & $0$ & $0$ & $0$ & $0$ & $0$ & $\frac{(1-r_N)^2+r_N^2}{2}$ & $\frac{1+2r_N-2r_N^2}{2}$
\\
$q_0$ & $0$ & $0$ & $0$ & $0$ & $0$ & $0$ & $1$
\\
\hline
\end{tabular}
}
\end{center}

\bigskip

When $N\to\infty$, the above matrix converges entry-wise to the matrix

\medskip

\begin{center}
\scalebox{0.85}{
\begin{tabular}{|c|ccccccccccccc|}
\hline
& $q_1$ & $q_2$ & $q_3$ & $q_4$ & $q_5$ & $q_6$ & $q_7$ & $q_8$ & $q_9$ & $q_{10}$ & $q_{11}$ & $q_{12}$ & $q_0$
\\
\hline
$q_1$ & 1 & 0 & 0 & 0 & 0 & 0 & 0 & 0 & 0 & 0 & 0 & 0 & 0
\\
$q_2$ & $0$ & $\frac{1}{2}$ & $0$ & $0$ & $0$ & $0$ & $0$ & $0$ & $0$ & $0$ & $0$ & $0$ &  $\frac{1}{2}$
\\
$q_3$ & $0$ & $0$ & $\frac{1}{2}$ & $0$ & $0$ & $0$ & $0$ & $0$ & $0$ & $0$ & $0$ & $0$ &  $\frac{1}{2}$
\\
$q_4$ & $0$ & $0$ & $0$ & $\frac{1}{2}$ & $0$ & $0$ & $\frac{1}{2}$ & $0$ & $0$ & $0$ & $0$ & $0$ &  $0$
\\
$q_5$ & $0$ & $0$ & $0$ & $0$ & $\frac{1}{4}$ & $0$ & $0$ & $0$ & $0$ & $0$ & $0$ & $0$ &  $\frac{3}{4}$
\\
$q_6$ & $0$ & $0$ & $0$ & $0$ & $0$ & $\frac{1}{4}$ & $0$ & $0$ & $0$ & $\frac{1}{2}$ & $\frac{1}{4}$ & $0$ &  $0$
\\
$q_7$ & $0$ & $0$ & $0$ & $r_N$ & $0$ & $0$ & $1-r_N$ & $0$ & $0$ & $0$ & $0$ & $0$ & $0$
\\
$q_8$ & $0$ & $0$ & $0$ & $0$ & $0$ & $0$ & $0$ & $\frac{1}{2}$ & $0$ & $0$ & $0$ & $0$ &  $\frac{1}{2}$
\\
$q_9$ & $0$ & $0$ & $0$ & $0$ & $0$ & $0$ & $0$ & $0$ & $\frac{1}{2}$ & $0$ & $0$ & $0$ &  $\frac{1}{2}$
\\
$q_{10}$ & $0$ & $0$ & $0$ & $0$ & $0$ & $\frac{r_N}{2}$ & $0$ & $0$ & $0$ & $\frac{1}{2}$ & $\frac{1-r_N}{2}$ & $0$ &  $0$
\\
$q_{11}$ & $0$ & $0$ & $0$ & $0$ & $0$ & $r_N^2$ & $0$ & $0$ & $0$ & $2r_N(1-r_N)$ & $(1-r_N)^2$ & $0$ &  $0$
\\
$q_{12}$ & $0$ & $0$ & $0$ & $0$ & $0$ & $0$ & $0$ & $0$ & $0$ & $0$ & $0$ & $\frac{(1-r_N)^2+r_N^2}{2}$ &  $\frac{1+2r_N-2r_N^2}{2}$
\\
$q_0$ & $0$ & $0$ & $0$ & $0$ & $0$ & $0$ & $0$ & $0$ & $0$ & $0$ & $0$ & $0$ & $1$
\\
\hline
\end{tabular}
}
\end{center}
\medskip


\subsection{Transition matrix  \texorpdfstring{${\bf \Pi}_N$}{TEXT} when \texorpdfstring{$r_N=0$}{TEXT} (no recombination)}

In the case when $r_N=0$ we expect that some of the transition probabilities become $0$. Specifically, once a group of gene copies from locus $i$ and locus $j$ are found on the same chromosome, they must remain on the same chromosome. 
For example, the transition probability from state $q_{11}$ to state $q_{10}$ must be $0$ (not necessarily vice-versa). 

Consider the transition probability from state $q_1$ to $q_7$. In state $q_1$ there are four chromosomes with loci on them, if a pair of these two chromosomes have the same parent chromosome it is possible we end up in state $q_7$. Specifically, when one of the chromosomes has a locus $i$ and one of the chromosomes has a locus $j$. This happens with probability $4\cdot \frac{1}{2N}=\frac{2}{N}$. 
When $r_N=0$ for all $N\in\mathbb{N}$, this matrix becomes

\medskip

\begin{center}
\scalebox{0.83}{
\begin{tabular}{|c|cccccc|}
\hline
& $q_1$ & $q_2$ & $q_3$ & $q_4$ & $q_5$ & $q_6$
\\
\hline
$q_1$ & $\frac{(N-1)(N-2)(N-3)}{N^3}$ & $\frac{(N-1)(N-2)}{2N^3}$ & $\frac{(N-1)(N-2)}{2N^3}$ & $\frac{2(N-1)(N-2)}{N^3}$ & $\frac{N-1}{4N^3}$ & $\frac{N-1}{2N^3}$
\\
$q_2$ & $\frac{(N-2)(N-3)(1-s_N)}{N^2}$ & $\frac{s_N(N-1)(N-2)}{2N^2}$ & $\frac{(N-2)(1-s_N)}{2N^2}$ & $\frac{2(N-2)(1-s_N)}{N^2}$ & $\frac{s_N(N-1)}{4N^2}$ & $\frac{1-s_N}{2N^2}$ 
\\
$q_3$ & $\frac{(N-2)(N-3)(1-s_N)}{N^2}$ & $\frac{(N-2)(1-s_N)}{2N^2}$ & $\frac{s_N(N-1)(N-2)}{2N^2}$ & $\frac{2(N-2)(1-s_N)}{N^2}$ & $\frac{s_N(N-1)}{4N^2}$ & $\frac{1-s_N}{2N^2}$
\\
$q_4$ & $\frac{(N-2)(N-3)(1-s_N)}{N^2}$ & $\frac{(N-2)(1-s_N)}{2N^2}$ & $\frac{(N-2)(1-s_N)}{2N^2}$ & $\frac{(N-2)(Ns_N-4s_N+3)}{2N^2}$ & $\frac{1-s_N}{4N^2}$ & $\frac{Ns_N-2s_N+1}{4N^2}$ 
\\
$q_5$ & $\frac{(N-2)(N-3)(1-s_N)^2}{N(N-1)}$ & $\frac{s_N(N-2)(1-s_N)}{2N}$ & $\frac{s_N(N-2)(1-s_N)}{2N}$ & $\frac{2(N-2)(1-s_N)^2}{N(N-1)}$ & $\frac{s_N^2(N-1)}{4N}$ & $\frac{(1-s_N)^2}{2N(N-1)}$ 
\\
$q_6$ & $\frac{(N-2)(N-3)(1-s_N)^2}{N(N-1)}$ & $\frac{(N-2)(1-s_N)^2}{2N(N-1)}$ & $\frac{(N-2)(1-s_N)^2}{2N(N-1)}$ & $\frac{(N-2)(1-s_N)(Ns_N-2s_N+1)}{N(N-1)}$ & $\frac{(1-s_N)^2}{4N(N-1)}$ & $\frac{(1-s_N)^2}{2N(2N-2)}+\frac{s_N^2(N-1)}{4N}$ 
\\
$q_7$ & $0$ & $0$ & $0$ & $0$ & $0$ & $0$
\\
$q_8$ & $0$ & $0$ & $0$ & $0$ & $0$ & $0$
\\
$q_9$ & $0$ & $0$ & $0$ & $0$ & $0$ & $0$
\\
$q_{10}$ & $0$ & $0$ & $0$ & $0$ & $0$ & $0$
\\
$q_{11}$ & $0$ & $0$ & $0$ & $0$ & $0$ & $0$
\\
$q_{12}$ & $0$ & $0$ & $0$ & $0$ & $0$ & $0$
\\
$q_0$ & $0$ & $0$ & $0$ & $0$ & $0$ & $0$
\\
\hline
\end{tabular}
}
\end{center}
\medskip

\begin{center}
\scalebox{0.74}{
\begin{tabular}{|c|ccccccc|}
\hline
& $q_7$ & $q_8$ & $q_9$ & $q_{10}$ & $q_{11}$ & $q_{12}$ & $q_0$
\\
\hline
$q_1$  & $\frac{2(N-1)(N-2)}{N^3}$ & $\frac{N-1}{N^3}$ & $\frac{N-1}{N^3}$ & $\frac{N-1}{N^3}$ & $\frac{N-1}{2N^3}$ & $\frac{1}{4N^3}$ & $\frac{4N-1}{4 N^2}$
\\
$q_2$ & $\frac{2(N-2)(1-s_N)}{N^2}$ & $\frac{s_N(N-1)}{N^2}$ & $\frac{1-s_N}{N^2}$ & $\frac{1-s_N}{N^2}$ & $\frac{1-s_N}{2N^2}$ & $\frac{s_N}{4N^2}$ & $\frac{2Ns_N-s_N+2}{4N}$
\\
$q_3$ & $\frac{2(N-2)(1-s_N)}{N^2}$ & $\frac{1-s_N}{N^2}$ & $\frac{s_N(N-1)}{N^2}$ & $\frac{1-s_N}{N^2}$ & $\frac{1-s_N}{2N^2}$ & $\frac{s_N}{4N^2}$ & $\frac{2Ns_N-s_N+2}{4N}$
\\
$q_4$ & $\frac{(N-2)(Ns_N-4s_N+3)}{2N^2}$ & $\frac{Ns_N-2s_N+1}{2N^2}$ & $\frac{Ns_N-2s_N+1}{2N^2}$ & $\frac{Ns_N-2s_N+1}{2N^2}$ & $\frac{Ns_N-2s_N+1}{4N^2}$ & $\frac{s_N}{4N^2}$ & $\frac{4N-1}{4N^2}$
\\
$q_5$ & $\frac{2(N-2)(1-s_N)^2}{N(N-1)}$ & $\frac{s_N(1-s_N)}{N}$ & $\frac{s_N(1-s_N)}{N}$ & $\frac{(1-s_N)^2}{N(N-1)}$ & $\frac{(1-s_N)^2}{2N(N-1)}$ & $\frac{s_N^2}{4N}$ & $\frac{(4-s_N)s_N}{4}$
\\
$q_6$ & $\frac{(N-2)(1-s_N)(Ns_N-2s_N+1)}{N(N-1)}$ & $\frac{s_N(1-s_N)}{N}$ & $\frac{s_N(1-s_N)}{N}$ & $\frac{(1-s_N)^2}{2N(N-1)}+\frac{s_N^2(N-1)}{2N}$ & $\frac{(1-s_N)^2}{2N(2N-2)}+\frac{s_N^2(N-1)}{4N}$ & $\frac{s_N^2}{4N}$ & $\frac{4-s_N^2}{4N}-\frac{(1-s_N)^2}{4N(N-1)}$
\\
$q_7$ & $\frac{(N-1)(N-2)}{N^2}$ & $\frac{N-1}{2N^2}$ & $\frac{N-1}{2N^2}$ & $\frac{N-1}{2N^2}$ & $\frac{N-1}{2N^2}$ & $\frac{1}{4N^2}$ & $\frac{4N-1}{4N^2}$
\\
$q_8$ & $\frac{(1-s_N)(N-2)}{N}$ & $\frac{s_N(N-1)}{2N}$ & $\frac{1-s_N}{2N}$ & $\frac{1-s_N}{2N}$ & $\frac{1-s_N}{2N}$ & $\frac{s_N}{4N}$ & $\frac{2Ns_N-s_N+2}{4N}$
\\
$q_9$ & $\frac{(1-s_N)(N-2)}{N}$ & $\frac{1-s_N}{2N}$ & $\frac{s_N(N-1)}{2N}$ & $\frac{1-s_N}{2N}$ & $\frac{1-s_N}{2N}$ & $\frac{s_N}{4N}$ & $\frac{2Ns_N-s_N+2}{4N}$
\\
$q_{10}$ & $\frac{(N-2)(1-s_N)}{N}$ & $\frac{1-s_N}{2N}$ & $\frac{1-s_N}{2N}$ & $\frac{s_N(N-1)}{2N}$ & $\frac{s_N(N-1)}{2N}$ & $\frac{s_N}{4N}$ & $\frac{4-s_N}{4N}$
\\
$q_{11}$ & $0$ & $0$ & $0$ & $0$ & $1-\frac{1}{N}$  & $\frac{1}{2N}$ & $\frac{1}{2N}$
\\
$q_{12}$ & $0$ & $0$ & $0$ & $0$ & $1-s_N$ & $\frac{s_N}{2}$ & $\frac{s_N}{2}$
\\
$q_0$ & $0$ & $0$ & $0$ & $0$ & $0$ & $0$ & $1$
\\
\hline
\end{tabular}
}
\end{center}

\medskip

When $N\to\infty$, the above matrix converges entry-wise to the matrix

\medskip

\begin{center}
\scalebox{0.8}{
\begin{tabular}{|c|ccccccccccccc|}
\hline
& $q_1$ & $q_2$ & $q_3$ & $q_4$ & $q_5$ & $q_6$ & $q_7$ & $q_8$ & $q_9$ & $q_{10}$ & $q_{11}$ & $q_{12}$ & $q_0$
\\
\hline
$q_1$ & 1 & 0 & 0 & 0 & 0 & 0 & 0 & 0 & 0 & 0 & 0 & 0 & 0
\\
$q_2$ & $1-s_N$ & $\frac{s_N}{2}$ & $0$ & $0$ & $0$ & $0$ & $0$ & $0$ & $0$ & $0$ & $0$ & $0$ &  $\frac{s_N}{2}$
\\
$q_3$ & $1-s_N$ & $0$ & $\frac{s_N}{2}$ & $0$ & $0$ & $0$ & $0$ & $0$ & $0$ & $0$ & $0$ & $0$ &  $\frac{s_N}{2}$
\\
$q_4$ & $1-s_N$ & $0$ & $0$ & $\frac{s_N}{2}$ & $0$ & $0$ & $\frac{s_N}{2}$ & $0$ & $0$ & $0$ & $0$ & $0$ &  $0$
\\
$q_5$ & $(1-s_N)^2$ & $\frac{s_N(1-s_N)}{2}$ & $\frac{s_N(1-s_N)}{2}$ & $0$ & $\frac{s_N^2}{4}$ & $0$ & $0$ & $0$ & $0$ & $0$ & $0$ & $0$ &  $\frac{4s_N-s_N^2}{4}$
\\
$q_6$ & $(1-s_N)^2$ & $0$ & $0$ & $s_N(1-s_N)$ & $0$ & $\frac{s_N^2}{4}$ & $s_N(1-s_N)$ & $0$ & $0$ & $\frac{s_N^2}{2}$ & $\frac{s_N^2}{4}$ & $0$ &  $0$
\\
$q_7$ & $0$ & $0$ & $0$ & $0$ & $0$ & $0$ & $1$ & $0$ & $0$ & $0$ & $0$ & $0$ & $0$
\\
$q_8$ & $0$ & $0$ & $0$ & $0$ & $0$ & $0$ & $1-s_N$ & $\frac{s_N}{2}$ & $0$ & $0$ & $0$ & $0$ & $\frac{s_N}{2}$
\\
$q_9$ & $0$ & $0$ & $0$ & $0$ & $0$ & $0$ & $1-s_N$ & $0$ & $\frac{s_N}{2}$ & $0$ & $0$ & $0$ & $\frac{s_N}{2}$
\\
$q_{10}$ & $0$ & $0$ & $0$ & $0$ & $0$ & $0$ & $1-s_N$ & $0$ & $0$ & $\frac{s_N}{2}$ & $\frac{s_N}{2}$ & $0$ &  $0$
\\
$q_{11}$ & $0$ & $0$ & $0$ & $0$ & $0$ & $0$ & $0$ & $0$ & $0$ & $0$ & $1$ & $0$ &  $0$
\\
$q_{12}$ & $0$ & $0$ & $0$ & $0$ & $0$ & $0$ & $0$ & $0$ & $0$ & $0$ & $1-s_N$ & $\frac{s_N}{2}$ &  $\frac{s_N}{2}$
\\
$q_0$ & $0$ & $0$ & $0$ & $0$ & $0$ & $0$ & $0$ & $0$ & $0$ & $0$ & $0$ & $0$ & $1$
\\
\hline
\end{tabular}
}
\end{center}
\medskip

\subsection{Transition matrix  \texorpdfstring{${\bf \Pi}_N$}{TEXT} when \texorpdfstring{$r_N=1/2$}{TEXT} (free recombination)} 

For free recombination i.e. $r_N=1/2$ loci select parent chromosomes independently of other loci on the same chromosome. The transition matrix when $r_N=1/2$ is the following.

\medskip

\begin{center}
\scalebox{0.8}{
\begin{tabular}{|c|cccccc|}
\hline
& $q_1$ & $q_2$ & $q_3$ & $q_4$ & $q_5$ & $q_6$
\\
\hline
$q_1$ & $\frac{(N-1)(N-2)(N-3)}{N^3}$ & $\frac{(N-1)(N-2)}{2N^3}$ & $\frac{(N-1)(N-2)}{2N^3}$ & $\frac{2(N-1)(N-2)}{N^3}$ & $\frac{N-1}{4N^3}$ & $\frac{N-1}{2N^3}$
\\
$q_2$ & $\frac{(N-2)(N-3)(1-s_N)}{N^2}$ & $\frac{s_N(N-1)(N-2)}{2N^2}$ & $\frac{(N-2)(1-s_N)}{2N^2}$ & $\frac{2(N-2)(1-s_N)}{N^2}$ & $\frac{s_N(N-1)}{4N^2}$ & $\frac{1-s_N}{2N^2}$ 
\\
$q_3$ & $\frac{(N-2)(N-3)(1-s_N)}{N^2}$ & $\frac{(N-2)(1-s_N)}{2N^2}$ & $\frac{s_N(N-1)(N-2)}{2N^2}$ & $\frac{2(N-2)(1-s_N)}{N^2}$ & $\frac{s_N(N-1)}{4N^2}$ & $\frac{1-s_N}{2N^2}$
\\
$q_4$ & $\frac{(N-2)(N-3)(1-s_N)}{N^2}$ & $\frac{(N-2)(1-s_N)}{2N^2}$ & $\frac{(N-2)(1-s_N)}{2N^2}$ & $\frac{(N-2)(Ns_N-4s_N+3)}{2N^2}$ & $\frac{1-s_N}{4N^2}$ & $\frac{Ns_N-2s_N+1}{4N^2}$ 
\\
$q_5$ & $\frac{(N-2)(N-3)(1-s_N)^2}{N(N-1)}$ & $\frac{s_N(N-2)(1-s_N)}{2N}$ & $\frac{s_N(N-2)(1-s_N)}{2N}$ & $\frac{2(N-2)(1-s_N)^2}{N(N-1)}$ & $\frac{(N-1)s_N^2}{4N}$ & $\frac{(1-s_N)^2}{2N(N-1)}$ 
\\
$q_6$ & $\frac{(N-2)(N-3)(1-s_N)^2}{N(N-1)}$ & $\frac{(N-2)(1-s_N)^2}{2N(N-1)}$ & $\frac{(N-2)(1-s_N)^2}{2N(N-1)}$ & $\frac{(N-2)(1-s_N)(Ns_N-2s_N+1)}{N(N-1)}$ & $\frac{(1-s_N)^2}{4N(N-1)}$ & $\frac{(1-s_N)^2}{2N(2N-2)}+\frac{s_N^2(N-1)}{4N}$ 
\\
$q_7$ & $0$ & $0$ & $0$ & $\frac{(N-1)(N-2)}{2N^2}$ & $0$ & $\frac{(N-1)}{4N^2}$
\\
$q_8$ & $0$ & $0$ & $0$ & $\frac{(1-s_N)(N-2)}{2N}$ & $0$ & $\frac{1-s_N}{4N}$
\\
$q_9$ & $0$ & $0$ & $0$ & $\frac{(1-s_N)(N-2)}{2N}$ & $0$ & $\frac{1-s_N}{4N}$
\\
$q_{10}$ & $0$ & $0$ & $0$ & $\frac{(N-2)(1-s_N)}{2N}$ & $0$ & $\frac{s_N(N-1)}{4N}$
\\
$q_{11}$ & $0$ & $0$ & $0$ & $0$ & $0$ & $\frac{N-1}{4N}$
\\
$q_{12}$ & $0$ & $0$ & $0$ & $0$ & $0$ & $\frac{1-s_N}{4}$
\\
$q_0$ & $0$ & $0$ & $0$ & $0$ & $0$ & $0$
\\
\hline
\end{tabular}
}
\end{center}

\medskip

\begin{center}
\scalebox{0.745}{
\begin{tabular}{|c|ccccccc|}
\hline
& $q_7$ & $q_8$ & $q_9$ & $q_{10}$ & $q_{11}$ & $q_{12}$ & $q_0$
\\
\hline
$q_1$  & $\frac{2(N-1)(N-2)}{N^3}$ & $\frac{N-1}{N^3}$ & $\frac{N-1}{N^3}$ & $\frac{N-1}{N^3}$ & $\frac{N-1}{2N^3}$ & $\frac{1}{4N^3}$ & $\frac{4N-1}{4 N^2}$
\\
$q_2$ & $\frac{2(N-2)(1-s_N)}{N^2}$ & $\frac{s_N(N-1)}{N^2}$ & $\frac{1-s_N}{N^2}$ & $\frac{1-s_N}{N^2}$ & $\frac{1-s_N}{2N^2}$ & $\frac{s_N}{4N^2}$ & $\frac{2Ns_N-s_N+2}{4N}$
\\
$q_3$ & $\frac{2(N-2)(1-s_N)}{N^2}$ & $\frac{1-s_N}{N^2}$ & $\frac{s_N(N-1)}{N^2}$ & $\frac{1-s_N}{N^2}$ & $\frac{1-s_N}{2N^2}$ & $\frac{s_N}{4N^2}$ & $\frac{2Ns_N-s_N+2}{4N}$
\\
$q_4$ & $\frac{(N-2)(Ns_N-4s_N+3)}{2N^2}$ & $\frac{Ns_N-2s_N+1}{2N^2}$ & $\frac{Ns_N-2s_N+1}{2N^2}$ & $\frac{Ns_N-2s_N+1}{2N^2}$ & $\frac{Ns_N-2s_N+1}{4N^2}$ & $\frac{s_N}{4N^2}$ & $\frac{4N-1}{4N^2}$
\\
$q_5$ & $\frac{2(N-2)(1-s_N)^2}{N(N-1)}$ & $\frac{s_N(1-s_N)}{N}$ & $\frac{s_N(1-s_N)}{N}$ & $\frac{(1-s_N)^2}{N(N-1)}$ & $\frac{(1-s_N)^2}{2N(N-1)}$ & $\frac{s_N^2}{4N}$ & $\frac{(4-s_N)s_N}{4}$
\\
$q_6$ & $\frac{(N-2)(1-s_N)(Ns_N-2s_N+1)}{N(N-1)}$ & $\frac{s_N(1-s_N)}{N}$ & $\frac{s_N(1-s_N)}{N}$ & $\frac{(1-s_N)^2}{2N(N-1)}+\frac{s_N^2(N-1)}{2N}$ & $\frac{(1-s_N)^2}{2N(2N-2)}+\frac{s_N^2(N-1)}{4N}$ & $\frac{s_N^2}{4N}$ & $\frac{4-s_N^2}{4N}-\frac{(1-s_N)^2}{4N(N-1)}$
\\
$q_7$ & $\frac{(N-2)(N-1)}{2N^2}$ & $\frac{N-1}{2N^2}$ & $\frac{N-1}{2N^2}$ & $\frac{N-1}{2N^2}$ & $\frac{N-1}{4N^2}$ & $\frac{1}{4N^2}$ & $\frac{4N-1}{4N^2}$
\\
$q_8$  & $\frac{(1-s_N)(N-2)}{2N}$ & $\frac{s_N(N-1)}{2N}$ & 
$\frac{1-s_N}{2N}$ & $\frac{1-s_N}{2N}$ & $\frac{1-s_N}{4N}$ & $\frac{s_N}{4N}$ & $\frac{2Ns_N-s_N+2}{4N}$
\\
$q_9$ & $\frac{(1-s_N)(N-2)}{2N}$ & $\frac{1-s_N}{2N}$ & $\frac{s_N(N-1)}{2N}$ & $\frac{1-s_N}{2N}$ & $\frac{1-s_N}{4N}$ & $\frac{s_N}{4N}$ & $\frac{2Ns_N-s_N+2}{4N}$
\\
$q_{10}$ & $\frac{(1-s_N)(N-2)}{2N}$ & $\frac{1-s_N}{2N}$ & $\frac{1-s_N}{2N}$ & $\frac{s_N(N-1)}{2N}$ & $\frac{s_N(N-1)}{4N}$ & $\frac{s_N}{4N}$ & $\frac{4-s_N}{4N}$
\\
$q_{11}$ & $0$ & $0$ & $0$ & $\frac{N-1}{2N}$ & $\frac{N-1}{4N}$  & $\frac{1}{4N}$ & $\frac{3}{4N}$
\\
$q_{12}$ & $0$ & $0$ & $0$ & $\frac{1-s_N}{2}$ & $\frac{1-s_N}{4}$  & $\frac{s_N}{4}$ & $\frac{3s_N}{4}$
\\
$q_0$ & $0$ & $0$ & $0$ & $0$ & $0$ & $0$ & $1$
\\
\hline
\end{tabular}
}
\end{center}

\medskip

When $N\to\infty$, the above matrix converges entry-wise to the matrix

\medskip

\begin{center}
\scalebox{0.77}{
\begin{tabular}{|c|ccccccccccccc|}
\hline
& $q_1$ & $q_2$ & $q_3$ & $q_4$ & $q_5$ & $q_6$ & $q_7$ & $q_8$ & $q_9$ & $q_{10}$ & $q_{11}$ & $q_{12}$ & $q_0$
\\
\hline
$q_1$ & 1 & 0 & 0 & 0 & 0 & 0 & 0 & 0 & 0 & 0 & 0 & 0 & 0
\\
$q_2$ & $1-s_N$ & $\frac{s_N}{2}$ & $0$ & $0$ & $0$ & $0$ & $0$ & $0$ & $0$ & $0$ & $0$ & $0$ &  $\frac{s_N}{2}$
\\
$q_3$ & $1-s_N$ & $0$ & $\frac{s_N}{2}$ & $0$ & $0$ & $0$ & $0$ & $0$ & $0$ & $0$ & $0$ & $0$ &  $\frac{s_N}{2}$
\\
$q_4$ & $1-s_N$ & $0$ & $0$ & $\frac{s_N}{2}$ & $0$ & $0$ & $\frac{s_N}{2}$ & $0$ & $0$ & $0$ & $0$ & $0$ &  $0$
\\
$q_5$ & $(1-s_N)^2$ & $\frac{s_N(1-s_N)}{2}$ & $\frac{s_N(1-s_N)}{2}$ & $0$ & $\frac{s_N^2}{4}$ & $0$ & $0$ & $0$ & $0$ & $0$ & $0$ & $0$ &  $\frac{(4-s_N)s_N}{4}$
\\
$q_6$ & $(1-s_N)^2$ & $0$ & $0$ & $s_N(1-s_N)$ & $0$ & $\frac{s_N^2}{4}$ & $s_N(1-s_N)$ & $0$ & $0$ & $\frac{s_N^2}{2}$ & $\frac{s_N^2}{4}$ & $0$ &  $0$
\\
$q_7$ & $0$ & $0$ & $0$ & $\frac{1}{2}$ & $0$ & $0$ & $\frac{1}{2}$ & $0$ & $0$ & $0$ & $0$ & $0$ & $0$
\\
$q_8$ & $0$ & $0$ & $0$ & $\frac{1-s_N}{2}$ & $0$ & $0$ & $\frac{1-s_N}{2}$ & $\frac{s_N}{2}$ & $0$ & $0$ & $0$ & $0$ &  $\frac{s_N}{2}$
\\
$q_9$ & $0$ & $0$ & $0$ & $\frac{1-s_N}{2}$ & $0$ & $0$ & $\frac{1-s_N}{2}$ & $0$ & $\frac{s_N}{2}$ & $0$ & $0$ & $0$ &  $\frac{s_N}{2}$
\\
$q_{10}$ & $0$ & $0$ & $0$ & $\frac{1-s_N}{2}$ & $0$ & $\frac{s_N}{4}$ & $\frac{1-s_N}{2}$ & $0$ & $0$ & $\frac{s_N}{2}$ & $\frac{s_N}{4}$ & $0$ &  $0$
\\
$q_{11}$ & $0$ & $0$ & $0$ & $0$ & $0$ & $\frac{1}{4}$ & $0$ & $0$ & $0$ & $\frac{1}{2}$ & $\frac{1}{4}$ & $0$ &  $0$
\\
$q_{12}$ & $0$ & $0$ & $0$ & $0$ & $0$ & $\frac{1-s_N}{4}$ & $0$ & $0$ & $0$ & $\frac{1-s_N}{2}$ & $\frac{1-s_N}{4}$ & $\frac{s_N}{4}$ &  $\frac{3s_N}{4}$
\\
$q_0$ & $0$ & $0$ & $0$ & $0$ & $0$ & $0$ & $0$ & $0$ & $0$ & $0$ & $0$ & $0$ & $1$
\\
\hline
\end{tabular}
}
\end{center}

It can be verified that the sum of the entries in each row is $1$ in all matrices in this subsection.

\section{Proofs}\label{A:Proofs}

\begin{proof}[Proof of Lemma \ref{L:ETi}]
By definition, $T_i=0$ almost surely under $\P_c$. Write the transition matrix in Figure \ref{M:single_locus} as
$\mathbf{\Lambda}_N=(p_{qk})_{q,k\in\{\text{coal}, \,\text{same},\, \text{diff}\}}$.
By a first step analysis, we obtain the following linear system of  equations: for $q\in\{\text{same}, \text{diff}\}$, 
\begin{equation}\label{Eq:1stMoment}
    \E_q[T_i]=\sum_{k\in\{\text{coal},\, \text{same},\, \text{diff}\}}p_{qk}\E_k[T_i+1] = 1+\sum_{k\in\{\text{same},\, \text{diff}\}}p_{qk}\E_k[T_i]
\end{equation}
and
\begin{equation}\label{Eq:2ndMoment}
    \E_q[T_i^2]=\sum_{k\in \{\text{coal},\, \text{same},\, \text{diff}\}} p_{qk}\E_k[(T_i+1)^2]=1+\sum_{k\in\{\text{same},\, \text{diff}\}}p_{qk}(\E_k[T_i^2]+2\E_k[T_i]).
\end{equation}
Solving the 4 equations given by \eqref{Eq:1stMoment} and \eqref{Eq:2ndMoment} for the 4 unknowns, we obtain Lemma \ref{L:ETi}.
\end{proof}

\medskip

\begin{proof}[Proof of Lemma \ref{L:Distribution_Ti}]
Following \cite[Chapter 6.2]{wakeley2009coalescent}, split $\mathbf{\Lambda}_N$ from \eqref{M:single_locus} as
$$
\mathbf{\Lambda}_{N}=A+\frac{1}{N}\,B,
$$
where $A\,:=\,\lim_{N\to\infty}\mathbf{\Lambda}_{N}$ contains the $O(1)$ terms of $\mathbf{\Lambda}_N$ and $B\,:=\,\lim_{N\to\infty}N(\mathbf{\Lambda}_{N}-A)$  the $O(\frac{1}{N})$ terms. Moreover, let $P\,:=\,\lim_{k\rightarrow\infty}A^k$ and $G\,:=\, P B P$. By
\cite[Lemma 1]{Mohle1998a},
\begin{align}
   \mathbb P_\text{same}(T_i>Nt) 
    =&\, 
    \left(0,1,0\right)\mathbf{\Lambda}_{N}^{[N t]} \left(0,1,1\right)^{T}\label{Pf:L2_1}
    \\
    \rightarrow&\,
     \left(0,1,0\right)P e^{tG} \left(0,1,1\right)^{T} \qquad \text{as }N\rightarrow\infty,\label{Pf:L2_3} 
\end{align}
for all $t\in(0,\infty)$. In \eqref{Pf:L2_1}-\eqref{Pf:L2_3}, the vector $ \left(0,1,0\right)$ is the initial distribution of the ancestral process and it corresponds to the case in which the sampled gene copies at generation $g=0$ are in the \textit{same} individual.  
The right hand side of \eqref{Pf:L2_3}  is equal to $e^{-t/(2-s)}\cdot \frac{2(1-s)}{2-s}$ when  $s_N=s$  and  equal to $e^{-t/2}$ when $s_N=\tilde{\sigma}/N$. This proves \eqref{R1:P_s_tail}. Changing $\left(0,1,0\right)$ in \eqref{Pf:L2_1}-\eqref{Pf:L2_3} to $\left(0,0,1\right)$ gives \eqref{R1:P_d_tail}. 

\end{proof}

\medskip

\begin{proof}[Proof of Theorem \ref{T:q12}] \label{proof:thm1}
Recall that ${\bf \Pi}_N$ is a $13\times 13$ Markovian matrix in \ref{A:Matrix}, whose last row and last column corresponding to the absorbing state $q_0$. We let
$\widetilde{\bf \Pi}_N=((\widetilde{\bf \Pi}_N)_{qk})_{q,k\in\mathcal{S}}$ be the $12\times 12$ \textit{sub-Markovian} matrix obtained by deleting the last row and the last column of ${\bf \Pi}_N$.

To derive the covariance between the coalescence times for two loci, we perform a first step analysis as in \cite[eqn.(15)]{king2018non}.     
We denote by $q_0$ the state in which coalescence occurred in one or both loci. Then $\E_{q_0}[T_iT_j]=0$, and for all $q\in \mathcal S$, the set of non-coalescent states, we have:   
    \begin{align} 
    \E_q[T_iT_j] &= \sum_{k\in \mathcal S\cup \{q_0\}}({\bf \Pi}_N)_{qk}\,\E_k[(T_i + 1)(T_j+ 1)] 
    \notag\\
    &= 1 + \sum_{k\in \mathcal S\cup \{q_0\}} ({\bf \Pi}_N)_{qk}\,\E_k[T_i] + \sum_{k\in \mathcal S\cup \{q_0\}}({\bf \Pi}_N)_{qk}\E_k[T_j] + \sum_{k\in \mathcal S\cup \{q_0\}} ({\bf \Pi}_N)_{qk}\,\E_k[T_iT_j] \notag\\
    &= \sum_{k\in \mathcal S} (\widetilde{\bf \Pi}_N)_{qk}\,\E_k[T_iT_j]+\E_q[T_i]+\E_q[T_j]-1,\label{E:expectation}
    \end{align}
    where the last equality follows since  $\E_q[T_i] = \sum_{k\in \mathcal S\cup \{q_0\}} ({\bf \Pi}_N)_{qk}\,\E_k[T_i]+1$ when $q\neq q_0$, and since $\E_{q_0}[T_iT_j]=0$. Note that $\bf \Pi_N$ is changed to $\widetilde{\bf \Pi}_N$ in the last equality.

Define the $12\times 1$ column vectors
\begin{align*}
\quad
 \vec{b}_N :=(\E_q[T_i]+\E_q[T_j] - 1)_{q\in \mathcal{S}} \qquad\text{and}
\qquad
 \vec{E}_N  :=(\E_q[T_iT_j])_{q\in \mathcal{S}}.
\end{align*} 
The vector $\vec b_N$  can be directly computed using Lemma \ref{L:ETi} and simple facts like
\begin{equation}\label{q12toss}
(\E_{q_{12}}[T_i],\;\E_{q_{12}}[T_j]) = (\E_{\text{same}}[T_i],\;\E_{\text{same}}[T_j]) \quad \text{and} \quad(\E_{q_{2}}[T_i],\;\E_{q_{2}}[T_j]) = (\E_{\text{same}}[T_i],\;\E_{\text{diff}}[T_j]).  
\end{equation}

Equation \eqref{E:expectation} can now be rewritten as
 the  matrix equation  
\begin{align} 
\vec{E}_N&=\widetilde{\bf \Pi}_N \cdot \vec{E}_N + \vec b_N \label{E:Mx}
\end{align}
which can be solved for 
the unknown vector $\vec{E}_N$ to obtain 
\begin{align}
\vec{E}_N &= \left({\bf I}_{12\times 12}-\widetilde{\bf \Pi}_N\right)^{-1}\,\vec b_N,\label{E:sol_E}
\end{align}
where ${\bf I}_{12\times 12}$ is the $12\times 12$ identity matrix. Note that the matrix ${\bf I}_{12\times 12}-\widetilde{\bf \Pi}_N$ is invertible because $\widetilde{\bf \Pi}_N$ is \textit{sub-Markovian} and hence all its eigenvalues have absolute values strictly less than 1. 
Using MATLAB code (\ref{A:Code}) we computed the RHS of \eqref{E:sol_E} and obtained $(\E_q[T_iT_j])_{q\in \mathcal{S}}$.

To obtain the expressions in \eqref{cov1}, \eqref{cov2}, \eqref{cov3}, and \eqref{cov4} we provide the asymptotic values of $\text{Cov}_q[T_i, T_j]=\E_q[T_iT_j]-\E_q[T_i]\cdot \E_q[T_j]$

We have already computed the values of $\E_q[T_i]\cdot\E_q[T_j]$ in Lemma \ref{L:ETi} and found that they are of order $O(N^2)$. 
The values of  $\lim_{N\to\infty}\E_{q}[T_iT_j]/N^2$ and of $\lim_{N\to\infty}\E_{q}[T_i]\cdot\E_{q}[T_j]/N^2$, for $q\in\mathcal{S}$, are listed in
\ref{A:expectation} tables \eqref{E:limN^2} and \eqref{E:limNN} respectively.
For states where $\lim_{N\to\infty}\text{Cov}_q[T_i, T_j]/N^2=0$ we  provide the value of $\lim_{N\to\infty}\text{Cov}_q[T_i, T_j]/N$ in \ref{E:Scenario4O(N)}.

\end{proof}

\medskip

\begin{proof}[Proof of Corollary \ref{C:q12q11}]
Similar to \eqref{q12toss}, it holds that
\begin{equation}\label{q12toss_Var}
(\textrm{Var}_{q_{12}}[T_i],\;\textrm{Var}_{q_{12}}[T_j]) = (\textrm{Var}_{s}[T_i],\;\textrm{Var}_{\rm same}[T_j])    
\end{equation}
Recall from Lemma \ref{L:ETi} that $\textrm{Var}_{s}[T_i]=(4-4s_N)N^2+(2-2s_N)N+2$. Hence
\begin{align*}
    \textrm{Corr}_{q_{12}}[T_i, T_j] = \frac{\textrm{Cov}_{q_{12}}[T_i, T_j]}{(4-4s_N)N^2+(2-2s_N)N+2}.
\end{align*}
The result now follows by using the values of $\textrm{Cov}_{q_{12}}[T_i, T_j]$ for each of the four scenarios from Theorem \ref{T:q12}. Note a similar argument can be used for states $q_{11}$ but we need to use $\text{Var}_{\rm diff}[T_i]$ instead of $\text{Var}_{\rm same}[T_i]$. Other initial conditions can be treated in the same way.
\end{proof}

\medskip

\begin{proof}[Proof of Theorem \ref{T:q12_extreme}]
We begin with state $q_{12}$. When $s_N=1$ for all $N\in\mathbb{N}$ (i.e. total selfing) and the initial gene copies are in a single individual, 
we only need to trace back a single ancestral individual for each generation in the past. This reduces to a simple Markov chain $(X_g)_{g\in\Z_+}$ where $X_g$ is the state of the chromosomes in {\it the ancestor} in generation $g$ in the past, whose transition is determined by the left column of Figure \ref{Fig:Chrom+loci}. In this case,
by a one-step analysis similar to that of \eqref{E:expectation} we have
\begin{align*}
    \E_{q_{12}}[T_iT_j] &= \frac{r_N^2+(r_N-1)^2}{2}\E_{q_{12}}[T_iT_j]+\E_\text{same}[T_i]+\E_\text{same}[T_j]-1 \\
    &= \frac{r_N^2+(r_N-1)^2}{2}\E_{q_{12}}[T_iT_j]+3.
\end{align*}
Upon solving, we obtain $\E_{q_{12}}[T_iT_j] = \frac{6}{1-2r_N^2+2r_N}$. Therefore,
\begin{align*}
\textrm{Cov}_{q_{12}}[T_i,T_j] = \,\frac{6}{1-2r_N^2+2r_N}-4.
\end{align*}


In a similar fashion, we can set up a one-step analysis of a Markov Chain to solve for the covariance of state $q_{11}$. However, we resort to an even simpler explanation. Note that \begin{align*}
    \E_{q_{11}}[T_iT_j] = \E_{q_{11}}[T_i^2]+O(N).
\end{align*}
Since $T_i$ and $T_j$ differ by $O(1)$ and $\E_{q_{11}}[T_i]$, $\E_{q_{11}}[T_j]$ are $O(N)$. Since $\E_{q_{11}}[T_i^2]$ is $O(N^2)$ we have that $\E_{q_{11}}[T_iT_j] \asymp \E_{q_{11}}[T_i^2]$ yielding that $\text{Cov}_{q_{11}}[T_i, T_j] \asymp \text{Var}[T_i]$.

Lastly by the same logic as for state $q_{11}$ we have:
\begin{align*}
    \E_{q_{5}}[T_iT_j] &= \frac{1}{4N}\E_{q_{12}}[T_iT_j]+\frac{N-1}{4N}\cdot \E_{q_5}[T_i, T_j]+\E_\text{same}[T_i]+\E_\text{same}[T_j]-1 
    \\
    \E_{q_{5}}[T_iT_j] &= \frac{4N}{3N+1} (3+\frac{6}{4N(1-2r_N^2+2r_N)})
\end{align*}

Therefore we obtain 
\begin{align*}
    \textrm{Cov}_{q_{5}}[T_iT_j] &= \E_{q_{5}}[T_iT_j] - 4  =\frac{2(2r_N-1)^2}{(3N+1)(1+2r_N-2r_N^2)}.
\end{align*}

When $r_N=0$ for all $N\in \mathbb N$ (i.e. no recombination) we note that loci belonging to the same chromosome will remain on the same chromosomes for each generation in the pass. This means we have that $T_i=T_j$. Thus $\text{Cov}_{q_{12}}[T_i, T_j] = \text{Var}_{q_{12}}[T_i]=(4-4s_N)N^2+(2-2s_N)N+2$ since the two gene copies at loci $i$ are in the same individual in state $q_{12}$. By the same argument we have that $\text{Cov}_{q_{11}}[T_i, T_j] = \text{Var}_{q_{11}}[T_i]=(4-4s_N+s_N^2)N^2+(2-3s_N)N+2$ since the two gene copies are in different individuals in state $q_{11}$. For state $q_5$ we first plug in $r_N=0$ into the matrix $\bf \Pi_N$ from Theorem \ref{T:q12}, and repeat the steps in the proof of Theorem \ref{T:q12} to obtain $\text{Cov}_{q_5}[T_i, T_j]\asymp 8(1-s)^2/9\cdot N^2$.
\end{proof}

\begin{proof}[Proof of Corollary \ref{C:q12_extreme}]
From Theorem \ref{T:q12_extreme} we have that when $s_N=1$:
\begin{align*}
\begin{cases}
    \text{Cov}_{q_{5}}[T_i, T_j]=\frac{2(2r_N-1)^2}{(3N+1)(1+2r_N-2r_N^2)}\\
    \text{Cov}_{q_{11}}[T_i, T_j]\asymp N^2\\
    \text{Cov}_{q_{12}}[T_i, T_j]=\frac{6}{1+2r_N-2r_N^2}-4
\end{cases}
\end{align*}

For state $q_5$ and $q_{12}$ we note that both pairs of loci $i$ and $j$ belong to the same individual. Thus we have that $\text{Var}_{q_5}[T_i]=\text{Var}_{q_5}[T_j]=\text{Var}_{q_{12}}[T_i]=\text{Var}_{q_{12}}[T_j]=2$ by setting $s_N=1$ in Lemma \ref{L:ETi}. Thus we have $\text{Corr}_{q_{5}}[T_i, T_j]=\frac{(2r_N-1)^2}{(3N+1)(1+2r_N-2r_N^2)}$ and $\text{Corr}_{q_{12}}[T_i, T_j]=\frac{(2r_N-1)^2}{(3N+1)(1+2r_N-2r_N^2)}$

For state $q_{11}$ we note that both pairs of loci $i$ and $j$ belong to the different individuals. Thus we have that $\text{Var}_{q_{11}}[T_i]=\text{Var}_{q_{11}}[T_j]\asymp N^2$ by setting $s_N=1$ in Lemma \ref{L:ETi}. Thus we have $\text{Corr}_{q_{11}}[T_i, T_j]\asymp \frac{N^2}{N^2} \asymp 1$.

\bigskip

Again from Theorem \ref{T:q12_extreme} we have that, when $r_N=0$,
\begin{align*}
\begin{cases}
\text{Cov}_{q_5}[T_i, T_j]\asymp 8(1-s_N)^2/9\cdot N^2\\
\text{Cov}_{q_{11}}[T_i, T_j]=(4-4s_N +s_N^2 )N^2 +(2-3s_N )N +2 \\
\text{Cov}_{q_{12}}[T_i, T_j]= (4-4s_N )N^2 +(2-2s_N )N +2
\end{cases}
\end{align*}

For state $q_5$ and $q_{12}$ we note that both pairs of loci $i$ and $j$ belong to the same individual. Thus we have that $\text{Var}_{q_5}[T_i]=\text{Var}_{q_5}[T_j]=\text{Var}_{q_{12}}[T_i]=\text{Var}_{q_{12}}[T_j]=(4-4s_N)N^2+(2-2s)N+2$ from Lemma \ref{L:ETi}. Thus we have $\text{Corr}_{q_{5}}[T_i, T_j]\asymp 2(1-s_N)/9$ and $\text{Corr}_{q_{12}}[T_i, T_j]=1$

For state $q_{11}$ we note that both pairs of loci $i$ and $j$ belong to the different individuals. Thus we have that $\text{Var}_{q_{11}}[T_i]=\text{Var}_{q_{11}}[T_j]=(4-4s_N+s_N^2)N^2+(2-3s_N)N+2$ from Lemma \ref{L:ETi}. Thus we have $\text{Corr}_{q_{11}}[T_i, T_j]=1$.

\end{proof}

\section{\texorpdfstring{$\lim_{N\to\infty}\E_{q}[T_iT_j]/N^2$}{TEXT} and \texorpdfstring{$\lim_{N\to\infty}\E_{q}[T_i]\cdot\E_{q}[T_j]/N^2$}{TEXT} for \texorpdfstring{$q\in\mathcal{S}$}{TEXT} }\label{A:expectation}
Set $\lambda_1=8r^2+26r+9$, $\lambda_1'=8\tilde{\rho}^2+26\tilde{\rho}+9$, and $\lambda_2=8\tilde{\rho}^2s^2-16\tilde{\rho}^2s-26\tilde{\rho}s+\lambda_1'$. The following table contains the values of $\lim_{N\to\infty}\E_q[T_iT_j]/N^2$.

\begin{align}
\begin{array}{c|c|c|c|c} \label{E:limN^2}
& \text{Scenario } (i) & \text{Scenario }(ii) & \text{Scenario }(iii) & \text{Scenario }(iv)\\
\hline
q_1 & \frac{4\lambda_1+8}{\lambda_1} & \frac{(s-2)^2(\lambda_2+2)}{\lambda_2} & 4 & (2-s)^2
\\[6pt]
q_2 & \frac{4\lambda_1+8}{\lambda_1} & \frac{2(1-s)(2-s)(\lambda_2+2)}{\lambda_2} & 4 & 2(1-s)(2-s)
\\[6pt]
q_3 & \frac{4\lambda_1+8}{\lambda_1} & \frac{2(1-s)(2-s)(\lambda_2+2)}{\lambda_2} & 4 & 2(1-s)(2-s)
\\[6pt]
q_4 & \frac{4\lambda_1+8}{\lambda_1} & \frac{(2-s)^2\cdot \lambda_2+(2-s)(4-s)}{\lambda_2} & 4 & (2-s)^2
\\[6pt]
q_5 & \frac{4\lambda_1+8}{\lambda_1} & \frac{4(1-s)^2(\lambda_2+2)}{\lambda_2} & 4 & 4(1-s)^2 
\\[6pt]
q_6 & \frac{4\lambda_1+8}{\lambda_1} & \frac{(2-s)^2\cdot\lambda_2-2\tilde{\rho}s^3+2\tilde{\rho}s^2+5s^2-4s+8}{\lambda_2} & 4 & (2-s)^2
\\[6pt]
q_7 & \frac{4\lambda_1+12}{\lambda_1} & \frac{(2-s)^2(\lambda_2+3)}{\lambda_2} & 4 & (2-s)^2
\\[6pt]
q_8 & \frac{4\lambda_1+12}{\lambda_1} & \frac{2(1-s)(2-s)(\lambda_2+3)}{\lambda_2} & 4 & 2(1-s)(2-s)
\\[6pt]
q_9 & \frac{4\lambda_1+12}{\lambda_1} & \frac{2(1-s)(2-s) (\lambda_2+3)}{\lambda_2} & 4 & 2(1-s)(2-s)
\\[6pt]
q_{10} & \frac{4\lambda_1+12}{\lambda_1} & \frac{(2-s)((2-s)\lambda_2-2\tilde{\rho}s^2+2\tilde{\rho}s+3s+6)}{\lambda_2} & 4 & (2-s)^2
\\[6pt]
q_{11} & \frac{4\lambda_1+8\tilde{\rho}+36}{\lambda_1} & \frac{(2-s)^2(\lambda_2-2\tilde{\rho}s+2\tilde{\rho}+9)}{\lambda_2} & 4 & (2-s)^2
\\[6pt]
q_{12} & \frac{4\lambda_1+8\tilde{\rho}+36}{\lambda_1} & \frac{2(1-s)(2-s)(\lambda_2-2\tilde{\rho}s+2\tilde{\rho}+9)}{\lambda_2} & 4 & \frac{2(s-1)(s-2)^2}{2s r^2-2s r+s -2} 
\end{array}
\end{align}

The following table contains the values of $\lim_{N\to\infty}\E_q[T_i]\cdot \E_q[T_j]/N^2$.
\begin{align}
\begin{array}{c|c|c|c|c|c} \label{E:limNN}
& d={\rm diff},\,s={\rm same} & \text{Scenario }(i) & \text{Scenario }(ii) & \text{Scenario }(iii) & \text{Scenario }(iv)\\
\hline
q_1 & (d,d) & 4 & (2-s)^2 & 4 & (2-s)^2
\\[6pt]
q_2 & (s,d) & 4 & 2(1-s)(2-s) & 4 & 2(1-s)(2-s)
\\[6pt]
q_3 & (d,s) & 4 & 2(1-s)(2-s) & 4 & 2(1-s)(2-s)
\\[6pt]
q_4 & (d,d) & 4 & (2-s)^2 & 4 & (2-s)^2
\\[6pt]
q_5 & (s,s) & 4 & 4(1-s)^2 & 4 & 4(1-s)^2
\\[6pt]
q_6 & (d,d) & 4 & (2-s)^2 & 4 & (2-s)^2
\\[6pt]
q_7 & (d,d) & 4 & (2-s)^2 & 4 & (2-s)^2
\\[6pt]
q_8 & (s,d) & 4 & 2(1-s)(2-s) & 4  & 2(1-s)(2-s)
\\[6pt]
q_9 & (d,s) & 4 & 2(1-s)(2-s) & 4 & 2(1-s)(2-s)
\\[6pt]
q_{10} & (d,d) & 4 & (2-s)^2 & 4 & (2-s)^2
\\[6pt]
q_{11} & (d,d) & 4 & (2-s)^2 & 4 & (2-s)^2
\\[6pt]
q_{12} & (s,s) & 4 & 4(1-s)^2 & 4 & 4(1-s)^2
\end{array}
\end{align}

For $q=q_{12}$ in Scenario $(iv)$, we have 
$\lim_{N\to\infty}\E_q[T_iT_j]/N^2\neq \lim_{N\to\infty}\E_q[T_i]\E_q[T_j]/N^2$ and so
$\text{Corr}[T_i, T_j]$ is of order $O(1)$. This makes sense  because there is an $O(1)$ probability that both loci coalesce in the first few generations as a result of selfing. 

Next, we note that $\lim_{N\to\infty}\E_q[T_iT_j]/N^2=\lim_{N\to\infty}\E_q[T_i]\E_q[T_j]/N^2$ for states $q_1, \dots, q_{11}$ in Scenario $(iv)$ and all 12 non-coalescence states for Scenario $(iii)$. In any of  these cases, $\text{Cov}_q[T_iT_j]$ is of order $O(N)$ and $\text{Corr}[T_i, T_j]$ is of order $O(1/N)$. 
In these cases, $r_N=O(1)$. 
When all loci four loci  are in the same individual there is a $O(1)$ (assuming $s_N=O(1))$ probability that selfing causes that both pairs coalesce before recombination occurs yielding $T_i=T_j$. Otherwise we need an event of  $O(1/N)$ probability to take place before all loci can be in the same individual. 

Since recombination is order $O(1)$, the probability that all four loci are in the same individual before coalescence is no greater than $O(1/N)$. Thus, recombination will only contribution to the $O(N^2)$ term of covariance when it is of order $O(1/N)$.

For Scenario $(i)$ and $(ii)$, where the recombination probability is of order $O(1/N)$,  there is probability of order $O(1)$ that after some generations backwards in time there are two chromosomes both with one locus $i$ (${\color{red} \bullet}$) and one locus $j$ (${\color{OliveGreen} \blacktriangle}$) before any coalescence event occurs. Since recombination is $O(1/N)$ there is a $O(1)$ probability that these two chromosomes coalesce before any other recombination event takes place. This means that $T_i=T_j$ with $O(1)$ probability, which gives $\text{Corr}[T_i, T_j]=O(1)$ for any initial state in Scenario $(i)$ and $(ii)$. 
Unless recombination probability is of $O(1/N)$ or we are in initial state $q_{12}$ with selfing probability $O(1)$ no other state has $\text{Corr}[T_i, T_j]=O(1)$, and so they all have $\text{Corr}[T_i, T_j]=O(1/N)$.

\section{\texorpdfstring{$\lim_{N\to\infty}\text{Cov}_q[T_i, T_j]/N$}{TEXT}  for suitable \texorpdfstring{$q\in\mathcal{S}$}{TEXT} }\label{A:smallcov}

In \ref{A:expectation} we explained how the $\text{Cov}_q[T_i, T_j]$ is of order $O(N)$ for states in $\mathcal S \setminus \{q_{12}\}$ in Scenario $(iv)$ and all states in $\mathcal S$ in Scenario $(iii)$. We now write down these $O(N)$ terms. 

\subsection{Covariance for Scenario \texorpdfstring{$(iv)$}{TEXT}} \label{E:Scenario4O(N)}

Let $R=(1-r)^2+r^2=2r^2-2r+1$ i.e. the probability that both chromosomes in an individual undergo recombination or both do not undergo recombination. 

\medskip

\resizebox{\textwidth}{!}{
$\begin{array}{c|c|}
\ & \lim_{N\to\infty}\textrm{Cov}_q[T_i, T_j]/N \textrm{ for Scenario (iv)} \\[6pt]
\hline 
q_1 & 0  \\[6pt]
q_2 & 0  \\[6pt]
q_3 & 0  \\[6pt]
q_4 & 0  \\[6pt]
q_5 & \frac{2s^3\left(1-s\right)\left(2R+s-2s R\right )}{\left(4-s^2\right)\left(2-s R\right)} \\[6pt]
q_6 & \frac{s^2\left(2-s \right)\left(\left(16r^5-40r^4+36r^3-14r^2+2r\right)s^3+\left(16r^4-40r^3+24r^2-2r-1\right)s^2-8r\left(2r^2-r-1\right)s+8r+4\right)}{4r\left(1-s\right)\left(s R-2\right)\left(rs+r- s-2\right)\left(2rs-s+2\right)} \\[6pt]
q_7 & 0 \\[6pt]
q_8 & \frac{s^2\left(1-s\right)\left(2R+s-2s R\right)}{\left(2-s\right)\left(2-s R\right)} \\[6pt]
q_9 & \frac{s^2\left(1-s\right)\left(2R+s-2s R\right)}{\left(2-s\right)\left(2-s R\right)} \\[6pt]
q_{10} & \frac{s\left(2-s\right)\left[\left(16r^5-40r^4+36r^3-14r^2+2r\right)s^4-\left(16r^3-24r^2+10r-1\right)s^3+\left(16r^4-48r^3+36r^2-4r-2\right)s^2+\left(24r^3-52r^2+36r-4\right)s-16r^3+24r^2-8r-8\right]}{4r\left(1-s\right)\left(s R-2\right)\left(rs+r- s-2\right)\left(2rs-s+2\right)}  \\[6pt]
q_{11} & \frac{\left(2-s\right)\left[\left(16r^5-40r^4+36r^3-14r^2+2r\right)s^5-\left(16r^4-24r^3+16r^2-6r+1\right)s^4+\left(32r^4-80r^3+80r^2-32r+4\right)s^3+8r\left(2r-1\right)s^2-\left(80r^2-80r+16\right)s+32r^2-32r+16\right]}{4r\left(1-s\right)\left(s R-2\right)\left(rs+r- s-2\right)\left(2rs-s+2\right)} \\
\end{array}$ }

\medskip


When $r_N=1/2$, the above table simplfies to the following:

\medskip

\begin{align}
\begin{array}{c|c|}
& \lim_{N\to\infty}\text{Cov}_q[T_i, T_j]/N \text{ for Scenario }(iv),\text{ when } r_N=1/2\\
\hline
q_1 & 0 
\\[6pt]
q_2 & 0
\\[6pt]
q_3 & 0
\\[6pt]
q_4 & 0
\\[6pt]
q_5 & \frac{4s^3(1-s)}{(4-s)(2-s)(s+2)} 
\\[6pt]
q_6 & \frac{4s^2(2-s)(2+s)}{(1-s)(3+s)(4-s)}
\\[6pt]
q_7 & 0
\\[6pt]
q_8 & \frac{2s^2(1-s)}{(4-s)(2-s)}
\\[6pt]
q_9 & \frac{2s^2(1-s)}{(4-s)(2-s)}
\\[6pt]
q_{10} & \frac{4s(2-s)(2+s)}{(1-s)(3+s)(4-s)} 
\\[6pt]
q_{11} & \frac{4(2-s)(2+s)}{(1-s)(3+s)(4-s)}
\\[6pt]
q_{12} & \infty
\end{array}
\end{align}

In Scenario $(iv)$ the order $O(N)$ term of $\text{Cov}_q[T_i, T_j]$ for states $q_1, q_2, q_3, q_4, q_7$ is again $0$. We also note that these states are the only ones that have in at least $3$ different individuals. Selfing increase covariance when all four loci belong to the same individual as explained in \ref{A:expectation}. Thus, the results make sense because for states with loci in $3$ or more individuals there is only a $O(1/N^2)$ probability that all four loci are in the same individual before coalescence. 

Further we note that the formulas for states $q_6, q_{10}$ and $q_{11}$. This is also anticipated since under free recombination these three states should give the same transition probabilities since it doesn't matter whether loci begin on the same chromosome or not. 

\subsection{Results for Scenario \texorpdfstring{$(iii)$}{TEXT}} \label{E:Scenario3O(N)}

\begin{align}
\begin{array}{c|c|}
& \lim_{N\to\infty}\text{Cov}_q[T_i, T_j]/N \text{ for Scenario }(iii)\\
\hline
q_1 & 0 
\\[6pt]
q_2 & 0
\\[6pt]
q_3 & 0
\\[6pt]
q_4 & 0
\\[6pt]
q_5 & 0 
\\[6pt]
q_6 & 0
\\[6pt]
q_7 & 0
\\[6pt]
q_8 & 0
\\[6pt]
q_9 & 0
\\[6pt]
q_{10} & 0 
\\[6pt]
q_{11} & \frac{2((1-r)^2+r^2)}{r(2-r)}
\\[6pt]
q_{12} & \frac{2((1-r)^2+r^2)(\tilde{\sigma}r(2-r)+(1-r)^2)}{r(2-r)}
\end{array}
\end{align}

In Scenario $(iii)$ we see that states in $\mathcal S - \{q_{11}, q_{12}\}$ give the order $O(N)$ of $\text{Cov}_q[T_i, T_j]$ is $0$. The only nonzero terms are $q_{11}$ and $q_{12}$. For states $q_{11}$ and $q_{12}$ the order $O(N)$ term of $\text{Cov}_q[T_i, T_j]$ arises from the $O(1/N)$ probability of a coalesce event before any recombination occurs. The expression $R:=(1-r)^2+r^2$ in these formulas arises from the sum of probabilities that chromosomes both undergo recombination or both don't undergo recombination since assuming that the chromosomes pick the same parent, this is twice the chance both loci pairs coalesce simultaneously.

\section{Code used}\label{A:Code}

The code for the simulations and computations for Theorem \ref{T:q12} can be found in this link: \url{https://github.com/dkogan7/coalescence-computations}.


\bibliographystyle{elsarticle-harv} 
\bibliography{Bio,Math}






\end{document}